\documentclass[10pt,a4paper]{article}
\usepackage{amsmath,amssymb,amsfonts,amsthm}
\usepackage{float,graphicx,color}
\usepackage[all,cmtip]{xy}
\usepackage{tikz}
\usetikzlibrary{matrix}

\newcommand{\rmd}{\mathrm{d}}

\newcommand{\rmC}{\mathrm{C}}

\newcommand{\rmH}{\mathrm{H}}

\newcommand{\rmL}{\mathrm{L}}

\newcommand{\rmP}{\mathrm{P}}

\newcommand{\bbd}{\mathbb{d}}

\newcommand{\bbE}{\mathbb{E}}

\newcommand{\bbK}{\mathbb{K}}

\newcommand{\bbM}{\mathbb{M}}
\newcommand{\bbN}{\mathbb{N}}

\newcommand{\bbR}{\mathbb{R}}
\newcommand{\bbS}{\mathbb{S}}

\newcommand{\bbV}{\mathbb{V}}

\newcommand{\bbZ}{\mathbb{Z}}

\newcommand{\frp}{\mathfrak{p}}

\newcommand{\frK}{\mathfrak{K}}

\newcommand{\calC}{\mathcal{C}}

\newcommand{\calR}{\mathcal{R}}
\newcommand{\calS}{\mathcal{S}}
\newcommand{\calT}{\mathcal{T}}

\newcommand{\calX}{\mathcal{X}}
\newcommand{\calY}{\mathcal{Y}}
\newcommand{\calZ}{\mathcal{Z}}

\newcommand{\transp}{\textsc{t}\,}

\newcommand{\rota}{\textsc{r}}

\newcommand{\id}{\mathrm{id}}

\DeclareMathOperator{\Div}{div}
\renewcommand{\div}{\Div}
\DeclareMathOperator{\rot}{rot}
\DeclareMathOperator{\grad}{grad}
\DeclareMathOperator{\curl}{curl}
\DeclareMathOperator{\defo}{def}
\DeclareMathOperator{\airy}{airy}

\newcommand{\curltcurl}{\curl \transp \curl}
\DeclareMathOperator{\hess}{hess}

\DeclareMathOperator{\sven}{sven}

\DeclareMathOperator{\Skew}{skew}

\DeclareMathOperator{\coker}{coker}

\DeclareMathOperator{\tr}{tr}

\DeclareMathOperator{\End}{End}

\newcommand{\ts}{\textstyle}

\newcommand{\bs}{{\scriptscriptstyle \bullet}}

\newcommand{\diff}{\mathrm{d}}
\newcommand{\rest}{\mathrm{r}}
\newcommand{\eval}{\mathrm{e}}
\newcommand{\trans}{\mathrm{t}}
\newcommand{\curv}{\mathrm{c}}

\DeclareMathOperator{\orient}{\mathrm{o}}

\DeclareMathOperator{\ext}{\mathsf{ext}}

\renewcommand{\diff}{\mathsf{d}}
\renewcommand{\rest}{\mathsf{r}}
\renewcommand{\eval}{\mathsf{e}}
\renewcommand{\trans}{\mathsf{t}}
\renewcommand{\curv}{\mathsf{c}}

\DeclareMathOperator{\poincare}{\frp}
\DeclareMathOperator{\koszul}{\frK}

\newcommand{\subcell}{\unlhd} 

\newcommand{\subcells}{\calS}

\newcommand{\lls}{[ \! [}
\newcommand{\rrs}{] \! ]}

\newcommand{\lsb}{ {\Big [} }
\newcommand{\rsb}{ {\Big ]} }

\newcommand{\poly}{\rmP}

\newcommand{\sym}{\mathrm{sym}}

\newcommand{\beq}{\begin{equation}}
\newcommand{\eeq}{\end{equation}}

\newcommand{\mapping}[4]{
\left\{
\begin{array}{rcl}
\displaystyle #1  &\to& #2\\
\displaystyle #3  &\mapsto & #4
\end{array} \right.
}

\newcommand{\twocases}[4]{
\left\{
\begin{array}{ll}
 #1 \ & \textrm{ if } #2 \\
 #3 \ & \textrm{ if } #4
\end{array}
\right.
}

\definecolor{myyellow}{rgb}{1, 0.85, 0.15}
\definecolor{myorange}{rgb}{0.85, 0.5, 0.15}

\definecolor{mypurple}{rgb}{0.425, 0.15, 0.425}

\definecolor{myred}{rgb}{0.7, 0.15, 0.15}
\definecolor{myblue}{rgb}{0.15, 0.15, 0.7}
\definecolor{mygreen}{rgb}{0.15, 0.7, 0.15}

\definecolor{mybrown}{rgb}{0.55, 0.3, 0.15}
\definecolor{mydarkbrown}{rgb}{0.275, 0.15, 0.075}

\definecolor{mylightgray}{rgb}{0.5,0.5,0.5}
\definecolor{mygray}{rgb}{0.333, 0.333, 0.333}
\definecolor{mydarkgray}{rgb}{0.167,0.167,0.167}

\newcounter{npoint}[section]

\newtheorem{theorem}{Theorem}[section]
\newtheorem{lemma}[theorem]{Lemma}
\newtheorem{corollary}[theorem]{Corollary}
\newtheorem{proposition}[theorem]{Proposition}

\theoremstyle{remark}
\newtheorem{example}{Example}[section]

\theoremstyle{definition}
\newtheorem{remark}{Remark}[section]
\newtheorem{definition}{Definition}[section]

\newcounter{quest}[exercise]

\newcounter{subquest}[quest]

\newcounter{pushlevel}[theorem]

\newcommand{\push}{
\stepcounter{pushlevel}
 \vspace{-1ex} \noindent \hspace{4.5ex} \begin{minipage}[t]{\textwidth*\real{0.99} - 4.5ex}
\mbox{}\hspace{-1ex}\rule[-1.5ex]{.1pt}{1.5ex}\rule{5.3ex}{.1pt}
\vspace{-.65ex}

}

\newcommand{\unpush}{
\end{minipage}
\noindent \mbox{}\hspace{4.2ex}\rule{.1pt}{1.5ex}\rule{5.3ex}{.1pt} \hspace{\stretch{1}}
\noindent
\addtocounter{pushlevel}{-1}

}

\newcounter{theopoint}[theorem]
\newcommand{\tpoint}{
\medskip
\stepcounter{theopoint}
\noindent (\roman{theopoint}) 
}

\title{
Finite Element Systems for vector bundles:\\  elasticity and curvature
}

\author{Snorre H. Christiansen\thanks{Department of Mathematics, University of Oslo, PO Box 1053 Blindern, NO 0316 Oslo, Norway. email: {\tt snorrec@math.uio.no.}}, 
Kaibo Hu\thanks{School of Mathematics,
Vincent Hall, 206 Church St. SE, University of Minnesota, Minneapolis MN 55455-0488, USA. email: {\tt khu@umn.edu.}}}

\date{}

\begin{document}

\maketitle

\begin{abstract}
We develop a theory of Finite Element Systems, for the purpose of discretizing sections of vector bundles, in particular those arizing in the theory of elasticity. In the presence of curvature we prove a discrete Bianchi identity. In the flat case we prove a de Rham theorem on cohomology groups. We check that some known mixed finite elements for the stress-displacement formulation of elasticity fit our framework. We also define, in dimension two, the first conforming finite element spaces of metrics with good linearized curvature, corresponding to strain tensors with Saint-Venant compatibility conditions. Cochains with coefficients in rigid motions are given a key role in relating continuous and discrete elasticity complexes. 
\end{abstract}

\bigskip


\section*{Introduction}
In this paper we first generalize the previously introduced framework of finite element systems (FES) \cite{Chr08M3AS}\cite{ChrHu18} so that it can treat, in particular, elasticity problems, and then provide concrete examples of finite element spaces, some old and some new, that fit the framework.

The general framework provides an approach to finite element discretizations of sections of vectorbundles, and complexes thereof, in particular differential forms with values in a given vector bundle. We make some comments about curvature, but most of the paper concerns the case of flat bundles. For applications in elasticity, the fiber can be identified as the space rigid motions.

In space dimension 2, one can distinguish between two differential complexes related to elasticity, which are formal adjoints of each other and give priority to stresses and strains, respectively. For the stress complex (\ref{eq:elastress}), we can check that the spaces defined in \cite{JohMer78} and \cite{ArnDouGup84} fit the framework. For the strain complex (\ref{eq:elastrain}), we introduce, also within the framework, some new finite element spaces. They model symmetric $2$-tensors (metrics) with a good Saint-Venant operator (linearized curvature).

In the obtained finite element complexes, rigid motion like degrees of freedom play a key role, at every index. The FES framework stresses this design principle, and relates it to the interpretation of elasticity in terms of rigid motion valued fields. 

Defining discrete spaces of metrics with good curvature in dimension 2 should be useful, in view of the importance of curved surfaces in several branches of mathematics, both pure and applied, whatever the distinction is. Such applications will be explored elsewhere. Another motivation for this work was to prepare the way for similar constructions in higher dimensions, especially 3 (with classical elasticity in mind) and 4 (with general relativity in mind). 

\paragraph{Previous work on FES.} 
Until now, the FES framework has been formulated in order to discretize de Rham complexes. It has been used to define finite element complexes of differential forms on polyhedral meshes \cite{Chr08M3AS}, accommodate upwinded finite element complexes containing exponentials \cite{Chr13FoCM}\cite{ChrHalSor14}, give new presentations of known elements \cite{ChrRap16} and to define elements with minimal dimension \cite{ChrGil16} under various constraints (such as containing given polynomials).

The regularity of the differential forms, in the above mentioned works, was $\rmL^2$ with exterior derivative in $\rmL^2$, and the defined finite elements were natural generalizations of,  in particular, the Raviart-Thomas-N\'ed\'elec (RTN) spaces \cite{RavTho77}\cite{Ned80}. The continuity is thus partial, and can be expressed as singlevaluedness of pullbacks to interfaces, corresponding, for vector fields, to continuity in either tangential or normal directions. 

In \cite{ChrHu18} we extended the FES framework so as to be able to impose stronger interelement continuity. For instance, for a conforming discretization of the Stokes equation, one would like to have spaces of fully continuous vector fields, satisfying a commuting diagram with respect to the divergence operator.  For de Rham sequences of higher regularity ($\rmH^1$ and, if desired, exterior derivative in $\rmH^1$), the required continuity can be expressed as singlevaluedness of all components of the differential form and, if desired, of its exterior derivative too, on interfaces. This led us, in \cite{ChrHu18}, to define FE complexes starting with the Clough-Tocher element, which is of class $\rmC^1$, instead of, say, Lagrange elements, which are of class $\rmC^0$. This provided the first conforming polynomial composite Stokes element in dimension 3 (and higher), with piecewise constant divergence and the degrees of freedom of \cite{BerRau85}. The latter seem to be the natural ones for lowest order approximations.

\paragraph{FE, MFE, FEEC, VEM.}
Recall Ciarlet's definition of a finite element (FE), as a space equipped with degrees of freedom \cite{Cia78}. 

For mixed finite element methods (MFE), pairs of finite element spaces that are compatible in the sense of Brezzi \cite{Bre74} should be identified. A particularly convenient tool for this purpose, has been the so-called commuting diagram property, see for instance \cite{RobTho91} page 552 and 570 and compare with \cite{BofBreFor13} \S 8.4 and \S 8.5. It can sometimes be derived from a commutation property of the interpolators associated with the degrees of freedom. In particular, in \cite{Ned86}, finite element $\grad-\curl-\div$ complexes were presented with degrees of freedom providing commuting diagrams. 

Arbitrary order finite element complexes of differential forms were defined in \cite{Hip99}. Whitney forms \cite{Whi57}\cite{Wei52} and the RTN spaces appear as special cases (lowest order -- arbitrary dimension, and arbitrary order -- low dimension, respectively). This connection between numerical methods and differential topology was first pointed out in \cite{Bos88}. Computational electromagnetics has been one of the main motivations \cite{Bos98}\cite{Ned01}\cite{Hip02}. Its interpretation in terms of differential forms is quite clearcut compared with the case for, say, computational fluid dynamics.

Systematically developing the theory of finite elements in terms of differential complexes equipped with commuting projections was advocated in \cite{Arn02}. Relating de Rham complexes to differential complexes appearing in elasticity, and viewing both as special cases of complexes of Hilbert spaces,  has lead to the finite element exterior calculus (FEEC) \cite{ArnFalWin06}\cite{ArnFalWin10}\cite{Arn18}. 

Stability of numerical methods is, in many cases, equivalent to the existence of projections onto the finite element spaces, satisfying commuting diagrams, and having appropriate boundedness or compactness properties \cite{ArnFalWin06}{\cite{ChrWin13IMA}. Uniformly bounded commuting projections can often be obtained from the interpolator associated with degrees of freedom, by a smoothing procedure \cite{Sch08}\cite{Chr07NM}\cite{ArnFalWin06}\cite{ChrWin08}\cite{ChrMunOwr11} (in chronological order of submission).

The FES framework downplays the role of degrees of freedom and stresses that, for a finite element space on a cell, there are implicit finite element spaces on the subcells. The claim is that making these spaces explicit has numerous benefits. For instance, it suggests defining FE spaces recursively, starting with low-dimensional cells, using some extension procedure, such as harmonic extension \cite{Chr08M3AS}. 

The latter technique is also basic to the Virtual Element Method (VEM) \cite{BeiEtAl13}. VEM $\grad-\curl-\div$ complexes are constructed in \cite{BeiEtAl16} ; for an interpretation as a FES of differential forms, valid in arbitrary dimension, see \S 2.1 in \cite{ChrGil16}. Notice that \cite{BeiEtAl16} is based on non-homogeneous harmonic extensions, which produces quite large spaces compared with the minimal ones \cite{ChrGil16}. Harmonic extensions with respect to a modified metric is used in \cite{Chr13FoCM}\cite{ChrHalSor14}. In \cite{ChrHal15} harmonic extensions with respect to a flat metric but a non-zero connection form is used. This illustrates that the FES framework can accomodate FE spaces constructed as solutions of certain PDEs, where it is not necessary to have explicit solution formulas in order to be able to compute with them.

\paragraph{FES principles: restrictions, differentials and cochains.}

In FES, spaces are posited on cells of all dimensions. Interelement continuity is expressed through certain \emph{restriction operators}, from spaces on cells to spaces on subcells. The spaces on the subcells can also be arranged in complexes, for certain \emph{induced differential operators}. The restriction operators and the induced differential operators must satisfy commutation relations to define a \emph{system}. For a given FES, a condition of \emph{compatibility} (Definition \ref{def:compat}), expressed as exactness properties of the restrictions and the induced differentials, ensures the existence of good degrees of freedom and in particular that the harmonic interpolator is well defined and commutes with the differential operators. 

In \cite{Chr08M3AS}, concerning de Rham complexes of low regularity ($\rmL^2$ with exterior derivative in $\rmL^2$), the relevant restriction operators were pullback by inclusion maps, and the induced differential operators were, again, the exterior derivative. In \cite{ChrHu18}, to ensure $\rmH^1$ regularity, the restriction operators could remember all components of the differential forms on the subcell, and possibly of the exterior derivative as well (if it is required to be $\rmH^1$). The induced differential operators now acted on all this information. Thus appeared some new vectorbundles on subcells, linked by differential operators that were not exactly the exterior derivative on the subcell: they retain additional information about the ambient cell. We therefore, for the framework, considered general complexes of spaces, not just complexes of differential forms.

General degrees of freedom are not essential in FES, but they are certainly accomodated and sometimes very convenient. On the other hand certain degrees of freedom are paramount for the development of the theory. For de Rham complexes these degrees of freedom are the integration of $k$-forms on the $k$-dimensional cells of the mesh. This gives rise to the de Rham map, which maps from differential forms to (real valued) cellular cochains ; it commutes with the differentials.

For elasticity complexes, we contend that \emph{cellular cochains with coefficients in rigid motions} are the right analogue. More precisely we introduce, for each cell of each dimension, a space which is naturally isomorphic to the space of rigid motions. Cochains with coefficients in these spaces form a complex.  A generalized \emph{de Rham map} from elasticity fields to such cochains with coefficients, is then defined and shown to commute. 

\paragraph{Finite element elasticity complexes.}
In dimension 2, for elasticity problems, the stress complex is implicitly used in \cite{JohMer78} and made explicit in \cite{ArnDouGup84}. Since then, many more discrete stress complexes have been defined, both conforming \cite{ArnWin02} and non-conforming \cite{ArnWin03}\cite{ArnFalWin07}. See \cite{HuZha15} and the references therein for more examples. Notice that the stress complexes in \cite{JohMer78} and \cite{ArnDouGup84} are composite and start with Clough-Tocher elements, whereas those in \cite{ArnWin02} are polynomial and start with Argyris elements. 

For the systematic design of discrete elasticity complexes, a link between de Rham complexes and elasticity complexes, known as the BGG construction \cite{Eas99}, has been developed \cite{ArnFalWin06IMA1}\cite{ArnFalWin06IMA2}. In \cite{ArnFalWin06IMA2}, known finite element de Rham complexes were tensorized with vectors in order to get vector valued de Rham complexes. Under a surjectivity condition (see page 58), the diagram chase then yielded new spaces for the elasticity complexes.

The finite element complexes defined here behave naturally with respect to the BGG diagram chase. That is, we can define finite element spaces for some interlinked vector valued de Rham complexes, such that the diagram chase at the discrete level works exactly as at the continuous level: isomorphisms at the continuous level correspond to isomorphisms at the  discrete level. One thus needs a large supply of discrete de Rham sequences, corresponding to different regularities, that match at different indices. While this can be dispensed of in the presentation of our elasticity elements, it was an important guiding principle towards their design and we have included remarks to this effect. 

Finite element de Rham complexes of higher regularity have been constructed \cite{FalNei13}\cite{Nei15}. See also \cite{GuzNei14MC}\cite{GuzNei14IMA} for related Stokes elements. A motivation behind \cite{ChrHu18}\cite{ChrHuHu18} was to have enough such sequences to address elasticity through diagram chasing.

For elasticity complexes and related BGG diagram chases, the case of $\rmC^\infty$ regularity is well established in the literature, but the choice of Sobolev spaces is often not explicit. We introduce several Sobolev spaces for our complexes, many of which are not simple tensor products. There are several possible choices for each smooth complex. To obtain the stress complex one can do the chase in (\ref{eq:chasestress})  or (\ref{eq:chasestressbis}). In the latter, the regularity is expressed with a  differential operator that acts on columns, whereas the differentials of the complex act by rows. For the strain complex we study two different regularities, corresponding to two different regularities in the diagram chase. Here also the regularity is expressed in terms of differential operators acting on columns as well as rows. A rationale behind our choice of Sobolev spaces is given in Remark \ref{rem:sobchoice}.

Another tool we have developed for the purposes of constructing elasticity elements are  Poincar\'e - Koszul operators for elasticity complexes \cite{ChrHuSan18}. The Cesaro - Volterra path integral is but one example.

\paragraph{Numerical methods for curvature problems.}

For the general framework, the main novelty here, compared with \cite{ChrHu18}, is that we introduce some generalizations of the de Rham maps. We are interested in discretizing sections of vector bundles. These are equipped with a connection. For applications in (linear) elasticity this connection is flat. We have implicitly linearized around the Euclidean metric, for which the Levi - Civita connection is flat, as well as other associated connections. But, for the definition of discrete vector bundles, we have also been mindful of situations where non-zero curvature is centre stage, and inspired by numerical methods developed for such problems:

\begin{description}
\item[Regge Calculus] (RC) \cite{Reg61}, a discrete approach to general relavivity, can be interpreted in a finite element context \cite{Chr04M3AS}\cite{Chr11NM}\cite{Chr15} and extended to higher orders \cite{Li18}. One then obtains, in dimension 2, strain complexes of low regularity: they end with discrete spaces containing measures, typically Dirac deltas at vertices. Here, on the contrary, the finite element fields are at least square integrable throughout the complexes. Notice also that, even though the Regge strain complexes are equipped with commuting projections, the degrees of freedom do not contain rigid motions. They are nevertheless natural: Regge metrics are determined by edge lengths (squared), at lowest order. 

A discrete Gauss-Bonnet theorem is valid for RC, as can be proved combinatorially, or by a smoothing technique as introduced in \cite{Chr15}. For the elements presented here, Gaussian curvature (linearized or not) is well defined by classical formulas, so that the Gauss-Bonnet theorem is immediate. While the regularity of Regge elements seems adapted to general relativity theory, higher regularity, as achieved here, could be important to other PDEs in Riemannian geometry, such as those treated in \cite{Aub98}\cite{ChoLuNi06}\cite{Cia13}. 

\item[Lattice Gauge Theory] (LGT) \cite{Wil74}, as extended to a finite element context \cite{ChrHal12JMP}, was also at the back of our minds during this work. In LGT one defines discrete connections and curvature, as well as a discrete Yang-Mills functional, but it is less clear what the discrete covariant exterior derivative and Bianchi identity should be. By contrast, for the discrete theory we develop here, both of these are explicit: the former is to some extent the basic building block, and the latter is obtained in Theorem \ref{theo:bianchi}.

For the flat case, for which we introduce the FES framework, we prove a variant of the de Rham theorem (e.g. \cite{Pra07} \S V.3.): the de Rham map induces isomorphisms on cohomology groups, from the space of gobal sections to the cochains with coefficients, see Theorem \ref{theo:derham}.

\item[Exponential fitting.] A role for covariant exterior derivatives in a FES context was also indicated in \cite{Chr13FoCM}, defining upwinded complexes of differential forms, generalizing exponentially fitted methods. Such constructions have applications to PDEs describing for instance convection diffusion problems or band gaps in photonic cristals. See Remark \ref{rem:convdiff} for further details. 
\end{description}

\paragraph{Sheaves,  differential geometry and formal theory of PDEs.}
The FES framework can be interpreted as a discrete sheaf theory (see Remark \ref{rem:presh}). For our purposes, sheaf theory can be summarized as a framework for gluing fields that are defined piecewise. This is done in Sobolev spaces with the help of, in particular, partitions of unity, and in finite element spaces, by imposing various forms of continuity through interfaces in a mesh. Sheaves help express this analogy. 

Discrete sheaf theory for various applications is also developed in \cite{Ghr14}\cite{Cur14}. Sheaf theory is usually expressed in the language of categories, which also originates in algebraic topology. It has been suggested that this language, sometimes derided  as ''abstract nonsense'' (following Steenrod, with more or less affection), could be useful not only for other branches of mathematics \cite{KasSch06}, but for sciences in general \cite{Spi14}.

What we refer to here as elasticity complexes are sometimes called Calabi complexes in differential geometry, whereas rigid motions correspond to Killing fields. Sheaf theoretic approaches to such, with applications to linearized gravity, have been considered, see \cite{Kha17}. 

Our discrete Bianchi identity, being combinatorially quite elementary, has precedents, of course. In particular similar identities can be found in synthetic differential geometry \cite{Koc96}. This axiomatic approach to geometry grew out of topos theory \cite{LanMoe94}\cite{KasSch06}, introduced by Grothendieck to provide the ideal ''double bed'' for the espousals of continuous and discrete theories. It remains to see if it will appeal to numerical analysis. 

There is a formal theory of partial differential equations, couched in homological algebra and developed by, in particular, Spencer \cite{Spe69}. See e.g. \cite{Pom94}\cite{Sei10}. Perhaps our paper can hint towards a corresponding theory for finite element spaces. For instance: could there be master finite elements for the Spencer sequences, from which all specific examples could be deduced by diagram chases?

Though not necessary in order to express the end product, namely new finite element spaces, the above mentioned philosophies (more specifically: homological approaches to differential operators and categorical approaches to geometry), were inspiring to us for developing them, even from a distance.

\paragraph{Outline.} The paper is organized as follows. In \S \ref{sec:flatvb} we develop a theory for discrete flat vector bundles. Results concerning vector bundles with curvature are relegated to \S \ref{sec:curvature}. In \S \ref{sec:fes} we detail the framework of finite element systems, for a given discrete flat vector bundle. In \S \ref{sec:elasticity} we provide background on elasticity, including relevant differential operators, differential complexes, the BGG diagram chase and Poincar\'e operators. In \S \ref{sec:fesstress} we detail FES for the stress complex, detailing the main example of the Johnson-Mercier element. In \S \ref{sec:fesstrain} we detail FES for the strain complex, with two different regularities, providing the new examples of finite elements for strain tensors (metrics) with compatible Saint Venant operator (linearized curvature).

\section{Discrete flat vector bundles \label{sec:flatvb}}

\subsection{Cellular complexes and cochains}
Let $\calT$ be a cellular complex. If $T, T'$ are cells in $\calT$ we write $T' \subcell T$ to signify that $T'$ is a subcell of $T$. Each cell of dimension at least one is supposed oriented. Given two cells $T$ and $T'$ in $\calT$, their relative orientation is denoted $\orient(T,T')$. It is $0$ unless $T'$ is a codimension one subcell of $T$, in which case it is $\pm 1$. The subset of $\calT$ consisting of $k$-dimensional cells is denoted $\calT^k$. The space of $k$-cochains is denoted $\calC^k(\calT)$ and consists of the maps from $\calT^k$ to $\bbR$. In other words a $k$-cochain assigns a real number to each cell of dimension $k$. Notice that $\calC^k(\calT)$ has a canonical basis indexed by $\calT^k$.

The cellular cochain complex is denoted $\calC^\bs(\calT)$. Its differential, also called the coboundary map, is denoted $\delta: \calC^k(\calT) \to \calC^{k+1}(\calT)$. Its matrix in the canonical basis is given by relative orientations.

All complexes considered in this paper are cochain complexes, in the sense that the differential increases the index.

\subsection{Discrete vectorbundles with connection}

\begin{definition}[Discrete vectorbundle with connection] \label{def:buncon} \mbox{}
\begin{itemize}
\item For each $T \in \calT$ we suppose that we have a vectorspace $L(T)$. We call this a \emph{discrete vectorbundle}. We call $L(T)$ the fiber of $L$ at $T$. 

\item Moreover, when $T'$ is a codimension 1 face of $T$, we suppose that we have an isomorphism $\trans_{TT'}: L(T') \to L(T)$, called the transport map from $T'$ to $T$. We call this a \emph{discrete connection}.
\end{itemize}  
\end{definition}

\begin{remark}[comparison with Lattice Gauge Theory] This setup is at variance with the choices made in Lattice Gauge Theory (LGT). LGT was initially defined for cubical complexes \cite{Wil74}. An analogue for simplicical complexes was developed in \cite{ChrHal12JMP}. There, a discrete vector bundle corresponds to a choice of vector space attached to vertices only, whereas we here associate a vectorspace to each cell, of every dimension, in $\calT$. Moreover, in LGT, a discrete connection is defined only on edges, as a choice of isomorphim between the vectorspaces attached to its two vertices ; here the discrete connection has many more variables. 
\end{remark}

\begin{definition}[Flatness of discrete connections] \label{def:flat}
Whenever $T''$ is a codim-2 face of $T$, if we let $T_0'$ and $T_1'$ be the two codimension 1 faces of $T$ which have  $T''$ as a common codimension 1 face, we require that the following diagram commutes:
\begin{equation}\label{eq:com}
\xymatrix{
L(T'') \ar[r]^{\trans_{T_1'T''}} \ar[d]^{\trans_{T'_0T''}} & L(T_1') \ar[d]^{\trans_{TT'_1}}\\
L(T_0') \ar[r]^{\trans_{TT'_0}} &  L(T)
}
\end{equation}
Or, if one prefers:
\begin{equation}
\trans_{TT'_0}\trans_{T'_0T''} = \trans_{TT'_1}\trans_{T_1'T''}.
\end{equation}
For reasons that will appear later we say that a discrete connection having this property is \emph{flat}.
\end{definition}

For instance we can choose a fixed vector space $V$ and let $\trans_{TT'} = \id_V$. This defines a discrete vectorbundle with a flat discrete connection. Discrete vectorbundles of this form, for a choice of vectorspace $V$, will be called trivial discrete vectorbundles. We may speak of the trivial discrete vectorbundle modelled on $V$, to make the choice of $V$ explicit.

In this setting one defines a cochain complex with coefficients in $L$, denoted $\calC^\bs(\calT, L)$, as follows:

\begin{definition}[Cochains with coefficients]\mbox{}
\begin{itemize}
\item The space $\calC^k(\calT, L)$ is nothing but $\Pi_{T \in \calT^k} L(T)$, whose elements will be families $(u(T))_{T \in \calT^k}$ such that for each $T \in \calT^k$, $u(T) \in L(T)$. Such a $u$ will be called a $k$-cochain with coefficients in $L$.

\item The  differential $\delta_\trans^k: \calC^k(\calT, L) \to \calC^{k+1}(\calT, L) $ is defined by: 
\begin{equation}\label{eq:discovextder}
(\delta_\trans^k u)(T) = \sum_{T' \subcell T} \orient(T,T') \trans_{TT'} u(T').
\end{equation}
The operator $\delta_\trans^\bs$ will be called the \emph{discrete covariant exterior derivative}.
\end{itemize}
\end{definition}

We notice that we do indeed have a complex:
\begin{lemma}\label{lem:dtdt0} The operators $\delta_\trans^\bs$ on $\calC^\bs(\calT, L)$, satisfy:
\begin{equation}
\delta_\trans^{k+1}\delta_\trans^k = 0.
\end{equation}
\end{lemma}
\begin{proof}
This is a direct consequence of (\ref{eq:com}), given what we know about relative orientations.
\end{proof}

This definition can be used in particular when a vector space $V$ has been chosen and  we let $L(T) = V$ for all $T \in \calT$ and $\trans_{TT'} = \id_V$ for all $T, T'$. This cochain complex will be denoted $\calC^\bs(\calT, V)$. With this notation we have in particular that $\calC^\bs(\calT, \bbR) = \calC^\bs(\calT)$, the standard cellular cochain complex introduced previously.

\begin{remark}[on invertibility of transport operators]
In Definition \ref{def:buncon} we could allow $\trans_{TT'} : L(T') \to L(T)$ to be just a morphism (not necessarily an isomorphism), and still get at complex $\calC^\bs(\calT, L)$ from (\ref{eq:discovextder}). However, we have in mind situations where the transport operators $\trans_{TT'}$ mimick the parallel transport associated with a connection on a vectorbundle, and these are isomorphisms. Compare with \cite{GelMan03} \S I.4.7.
\end{remark}

\subsection{Transport along some paths within a cell}
We notice that when the discrete connection is flat, transport along paths within a given cell just depends on the endpoints. We will use the following more precise statement:
\begin{lemma}\label{lem:pathindep} 
Suppose $T, T' \in \calT$, and that $T'$ is a codimension-$k$ face of $T$, for some $k \geq 2$. Then  all sequences $T' \subcell T_0 \subcell \cdots \subcell T_{k-1} \subcell T$, where each term in this sequence is a codimension-$1$ subcell of the next, give the same map $\trans_{TT_{k-1}} \cdots \trans_{T_0 T'}$ from  $L(T')$ to $L(T)$. This map will be denoted $\trans_{T T'}$.
\end{lemma}

We could therefore make the following alternative description of a discrete vectorbundle with connection:

\begin{definition}
[equivalent definition of flat discrete vectorbundles]
For each $T \in \calT$ we suppose that we have, as before, a vectorspace $L(T)$. Moreover, when $T' \subcell T$, we suppose we have an isomorphim $\trans_{TT'}: L(T') \to L(T)$. We require that $\trans_{TT} = \id_{L(T)}$ and also that whenever $T'' \subcell T' \subcell T$ we have $\trans_{T T''} = \trans_{TT'}\trans_{T'T''}$. 
\end{definition}

\subsection{Discrete gauge transformations \label{sec:gauge}}

Suppose we have, for each $T\in \calT$, two choices of vectorspaces denoted $L(T)$ and $L'(T)$. For $T'$ a codim-1 face of $T$ we suppose that we have transport maps $\trans_{TT'} : L(T') \to L(T)$ as well as $\trans_{TT'}' : L'(T') \to L'(T)$. Under these circumstances we define an isomorphism from $(L, \trans)$ to $(L', \trans')$ to be a family of isomorphims $\theta_T : L(T) \to L'(T)$, one for each $T\in \calT$, such that:
\begin{equation}\label{eq:thetacom}
\theta_T \trans_{TT'} = \trans_{TT'}' \theta_{T'}.
\end{equation}
\begin{lemma}
Under the above circumstances $\theta$ induces an isomorphism of complexes $\calC^\bs(\calT, L) \to \calC^\bs(\calT, L') $, defined simply by:
\begin{equation}
\theta : (u(T))_{T\in \calT^k} \mapsto (\theta_T u(T))_{T\in \calT^k}.
\end{equation}
\end{lemma}
\begin{proof}
Bijectivity is obvious. We prove that $\theta$ is a cochain morphism (i.e. commutes with the differentials). We let $\delta_{\trans'}^\bs$ be the cochain-map of $\calC^\bs(\calT, L')$, while $\delta_\trans^\bs$ denotes that of $\calC^\bs(\calT,L)$. We have, for $u \in \calC^k(\calT, L)$:
\begin{align}
(\delta_{\trans'}^{ k} \theta u)(T) & = \sum_{T'\subcell T} \orient(T,T') \trans_{TT'}' \theta_{T'} u(T'),\\
& = \sum_{T'\subcell T} \orient(T,T') \theta_T \trans_{TT'} u(T'),\\
& = (\theta \delta_\trans^k u) (T),
\end{align}
as required.
\end{proof}
A family of isomorphisms $\theta$ as above may also be referred to as a \emph{discrete gauge transformation}. This terminology is used in particular when we have one discrete vectorbundle $L$, but with two different choices of discrete connections $\trans$ and $\trans'$; in this case $\theta_T$ will be an automorphism of $L(T)$, for each $T$.

\begin{lemma}\label{lem:cohlt}
For each cell $T$ in $\calT$, the complex $\calC^\bs(\subcells (T), L)$ (with coefficients) is isomorphic to $\calC^\bs(\subcells (T), L(T))$ (with constant fiber). Explicitely, for each subcell $S \subcell T$ we let $\theta_S : L(S) \to L(T)$ be the map $\trans_{TS}$. Then $\theta$ gives a gauge transformation, from $(L, \trans)$ to $(L(T), \id)$.

\end{lemma}
\begin{proof}
The requirement is that the following commutation relation holds for subcells $S, S'$ of $T$ such that $S' \subcell S$:
\begin{equation}
\theta_S \trans_{SS'} = \trans_{SS'}' \theta_{S'}.
\end{equation}
This can also be written:
\begin{equation}
\trans_{TS} \trans_{SS'} = \id_{L(T)} \trans_{TS'}.
\end{equation}
This identity holds according to Lemma \ref{lem:pathindep}.
\end{proof}

\begin{corollary}\label{cor:cohlt}
The sequence $\calC^\bs(\subcells (T), L)$ is exact, except at index $0$, where the kernel is isomorphic to $L(T)$.
\end{corollary}

The preceding corrollary is local, in that it concerns a single cell $T \in \calT$. The cohomology of the global space $\calC^\bs(\calT, L)$ could be different from that of $\calC^\bs(\calT, L(T))$ (for any choice of a fixed $T \in \calT$).

\begin{remark}[classification of flat discrete vector bundles]
Flat vector bundles over a manifold $M$, modulo gauge transformations, correspond to representations of the fundamental group of $M$, modulo conjugacy. For a precise statement, see for instance Theorem 13.2 in \cite{Tau11} (in the context of principal bundles). This seems to carry over to the discrete setting. The fundamental group of a cellular complex can be replaced by the so-called edgepath group (described in Chapter 3 in \cite{Spa95} for the case of simplicial complexes). 
\end{remark}

\section{Discrete vector bundles with curvature \label{sec:curvature}}

\subsection{Cubes in the barycentric refinement}
Consider now an $n$-dimensional cube. For definiteness we consider the unit cube $[0,1]^n$ in $\bbR^n$ and denote it as $S$. We let $(e_i)_{i \in \lls 1, n \rrs}$ be the canonical basis of $\bbR^n$.

The vertices of $S$ can be indexed by the subsets of $\lls 1, n \rrs$. For any subset $U$ of $\lls 1, n \rrs$, we let $p_U$ be the vertex of $S$ defined by:
\begin{equation}
p_U = \sum_{i \in U} e_i.
\end{equation}

Any face $S'$ of $S$ is uniquely determined by two vertices $p_U$ and $p_V$, with $U \subseteq V$,  such that the vertices of $S'$ are exactly those of the form $p_W$ for $U \subseteq W \subseteq V$. Then we also have:
\begin{equation}
S' = \{ p_U + \sum_{i \in V \setminus U} \xi_i e_i \ : \ \forall i \in V \setminus U  \quad \xi_i \in [0, 1] \}.
\end{equation}
Subsets of $\lls 1, n \rrs$ are partially ordered by inclusion, and this uniquely determines a partial ordering of the vertices. Then $p_U$ is the smallest vertex of $S'$ and $p_V$ the largest.

Let $T$ be a simplex. For each face $T'$ of $T$ we let $b_{T'}$ be the isobarycenter of $T'$ or, more generally, a point in the interior of $T'$ (so a barycenter with respect to some strictly positive weights), referred to as the inpoint of $T'$. Recall that the barycentric refinement of $T$ is the simplicial complex whose $k$-dimensional simplices are those of the form $[b_{T_0}, b_{T_1}, \ldots , b_{T_k} ]$ such that the $T_i$ are two by two distinct subsimplices of $T$ satisfying $T_0 \subcell T_1 \subcell \ldots \subcell T_k$. We call such subsimplices of $T$ barycentric simplices.

The barycentric refinement may be coarsened as follows: for any two subsimplices $T''$ and $T'$ of $T$ such that $T'' \subcell T'$, we consider the cell $S(T'', T')$ which is the union of all the barycentric simplices $[b_{T_0}, b_{T_1}, \ldots , b_{T_k} ]$ such that $T'' \subcell T_0$ and $T_k \subcell T'$.

When we start with a simplex $T$, the cells of the form $S(T'', T')$ form a cellular complex, where each cell is, combinatorially, a cube.  The same holds true if the cell $T$ we start with is a cube. We may consider that $b_{T''}$ is the smallest vertex of $S(T'', T')$ and that $b_{T'}$ is the largest. The vector $b_{T'} - b_{T''}$ points towards the center of $T$. When $T$ is an $n$-dimensional simplex, this procedure will divided it into $(n+1)$ cubes of dimension $n$, each one of the form $S(T', T)$, where $T'$ is a vertex of $T$. We call this the cubical refinement of $T$.

\subsection{Discrete curvature and Bianchi identity}

If we relax condition (\ref{eq:com}) we model the parallel transport associated with connections with curvature, as opposed to flat connections. We still define the discrete covariant exterior derivative by (\ref{eq:discovextder}). We no longer have $\delta_\trans^{k+1}\delta_\trans^k = 0$, as in Lemma \ref{lem:dtdt0}. When we compute $\delta_\trans^{k+1}\delta_\trans^k$ we get an operator $\calC^k(\calT, L) \to \calC^{k+2}(\calT, L)$, which one would like to interpret in terms of a curvature.

\begin{definition}[Discrete curvature] \label{def:curv}
We suppose that $T''$ and $T$ are to cells of $\calT$, such that $T''$ is a codimension 2 face of $T$, and we let $T_0'$ and $T_1'$ be the two codimension 1 faces of $T$ which have  $T''$ as a common codimension 1 face. The curvature of $\trans$ is then $\curv_\trans$ defined by:
\begin{equation}\label{eq:disccur}
\curv_\trans(T, T'') = \pm (\trans_{TT'_0}\trans_{T'_0T''}- \trans_{TT'_1}\trans_{T_1'T''}).
\end{equation}
which is associated with the square $S(T'', T)$ (element of the cubical refinement of $\calT$), whose set of vertices is associated with $\{T'',T_0', T_1', T \}$. The sign in this definition is given by the orientation of the square, which is chosen such that the orientations of the two transverse cells $T''$ and $S(T'', T)$ induce the orientation of $T$.
\end{definition}

The definitions were chosen so as to have the trivial:
\begin{lemma}
A discrete connection is flat according to Definition \ref{def:flat} iff its curvature according to Definition \ref{def:curv} is $0$. 
\end{lemma}
 
Definition \ref{def:curv} also gives:
\begin{proposition}[Curvature and the discrete covariant derivative]  With the notations of the Definition \ref{def:curv} we have:
\begin{equation}
(\delta_\trans^{k+1}\delta_\trans^k u)(T) =\pm  \curv_\trans(T, T'') u(T'').
\end{equation}
\end{proposition}

We have a Bianchi identity in this setting, which we now detail. Recall that the usual Bianchi identity says that the covariant exterior derivative of the curvature 2-form, which is a priori a certain endomorphism valued 3-form, is $0$.  In this identity, the relevant covariant exterior derivative is the one associated with the induced connection on the bundle of endomorphisms. The discrete identity will assert that certain linear operators attached to the 3-dimensional cubes, in the cubical refinement, is $0$.

\begin{definition}[Discrete bundle of endomorphisms and its connection] \mbox{}
\begin{itemize}
\item For $T\in \calT$ we consider the previously introduced cubical refinement of $T$, whose $k$-dimensional cells are (combinatorial) cubes of the form $S(T'', T')$, for subcells $T''$ and $T'$ of $T$, where $T''$ has codimension $k$ in $T'$. We let $\calS$ denote the cubical refinement of $\calT$. We then define $\End(L)$ to be the discrete vectorbundle on $\calS$, whose fiber at the cube $S(T'', T')$ is the space of linear maps from $L(T'')$ to $L(T')$.

\item The discrete vectorbundle $\End(L)$ on $\calS$ inherits a discrete connection from the discrete connection of $L$ on $\calT$, as follows. 

Consider a $k$-dimensional cube $S = S(T'', T')$, where we say that $b_{T''}$ is the smallest vertex and $b_{T'}$ is the largest. When $S'$ is a codimension $1$ face of $S$ there are two possibilities : either $b_{T''}$ is the smallest vertex of $S'$ and then we let $b_{T_0}$ be the largest, or $b_{T'}$ is the largest vertex of $S'$ and then we let $b_{T_0}$ be the smallest. We define the transport operator on $\End(L)$ through:
\begin{equation}
\mapping{\End(L)(S')}{\End(L)(S)}{u}{\twocases{\trans_{T' T0} \circ u}{b_{T''} \in S'}{u \circ \trans_{T_0 T''}}{b_{T'} \in S'}}
\end{equation}
\end{itemize}
\end{definition}

The spaces $\calC^\bs(\calS, \End(L))$ are defined as before and the discrete covariant exterior derivative linking these spaces, is defined as in (\ref{eq:discovextder}) from the given induced discrete connection on $\End(L)$.

\begin{remark} A discrete connection for the discrete bundle $L$ over $\calT$ is thus an element of $\calC^1(\calS, \End(L))$, where the element of $\End(L)$  attached to each edge of $\calS$ is \emph{bijective}. Compare with the fact that, in the continuous setting, the \emph{difference} between two connections is an endomorphism valued 1-form.
\end{remark}

\begin{theorem}[discrete Bianchi identity]\label{theo:bianchi}
The curvature of $(L, \trans)$, which is defined as an element of $\calC^2(\calS, \End(L))$ by (\ref{eq:disccur}), has a covariant exterior derivative (element of $\calC^3(\calS, \End(L))$) which is zero.
\end{theorem}
\begin{proof}
That the discrete covariant exterior derivative of the curvature is zero expresses that for each 3-dimensional cube $S(T'',T')$ a certain linear map from $L(T'')$ to $L(T')$ is zero. This linear map is a sum of maps of the form:
\begin{equation}\label{eq:transcube}
\pm \trans_{T'T_1}\trans_{T_1T_0}\trans_{T_0T''},
\end{equation}
where the cells $T'' \subcell T_0 \subcell T_1 \subcell T'$ represent vertices of the cube, in increasing order. The sum consists of two such contributions from the curvature of each of the six faces of the cube. We thus get twelve maps of the form (\ref{eq:transcube}). They cancel two by two ; in fact each map $\trans_{T'T_1}\trans_{T_1T_0}\trans_{T_0T''}$ appears twice in the sum, with different signs.
\end{proof}

\begin{remark}[consistency] One would like the discrete covariant exterior derivative to be in some sense consistent with a continuous one. Recall that the coboundary operator acting on (realvalued) simplicial cochains is isomorphic to the exterior derivative acting on Whitney form, via the de Rham map. One would like a similar interpretation of the discrete covariant exterior derivative. 

For cellular complexes, an analogue of Whitney forms was provided in \cite{Chr08M3AS}, by solving recursively, the PDE system $\rmd^\star \rmd u = 0$ and $\rmd^\star u = 0$, or a discrete analogue. One motivation for identifying induced operators on subcells is to extend this construction to other differential complexes, where one wants to find preimages of cochains with coefficients. This connects with a broader theme of defining finite elements as solutions of local PDEs (possibly discretized at a subgrid scale).

This also raises the question of, to which extent, from a discrete vector bundle, one can reconstruct a continuous vector bundle. In the flat case this seems unproblematic. In the presence of curvature, a condition of small curvature might be necessary, to mimick that fibers vary continuously in the continuous setting. For instance, one could require that the maps in Lemma \ref{lem:pathindep} should be close to each other, in some sense. As interesting and perhaps simpler special cases, one could consider the reconstruction of line bundles, and bundles over two-dimensional manifolds.

See \cite{KnoPin16} for similar considerations.
\end{remark}

\begin{remark}[gauge transformations and curvature] 
Discrete gauge transformations are defined as in \S \ref{sec:gauge}, also in the presence of curvature.

Notice that the discrete curvature transforms naturally under discrete gauge tranformations. Indeed, consider a discrete connection $\trans$ for $L$ and a discrete connection $\trans$ for $L'$, as well as gauge transformations $\theta : L \to L'$ such that (\ref{eq:thetacom}) holds. Then we have:
\begin{equation}
\theta_T \curv_\trans(T, T'') = \curv_{\trans'}(T, T'') \theta_{T''}.
\end{equation}

Also, given a gauge transformation $\theta$ from $(L, \trans)$ to $(L', \trans')$, there is an induced gauge transformation from $\End(L)$ to $\End(L')$, equipped with their induced connections. For a cube $S(T'', T')$ it maps $u : L(T'') \to L(T')$ to $\theta_{T'} u \theta_{T''}^{-1} : L'(T'') \to L'(T')$. The above transformation of curvature can be seen as a special case.
\end{remark}

\begin{remark}[Chern classes] 
A motivation for developing a discrete Bianchi identity is to develop a corresponding notion of Chern classes. A first Chern class is defined in this spirit in \cite{Phi85}. The topic is discussed in Chapter 7 of \cite{Sch18}.
\end{remark}

\section{Finite element systems \label{sec:fes}}

\begin{definition} We fix a flat discrete vector bundle $L$ on $\calT$ in the above sense. A \emph{finite element system} on $\calT$ consists of the following data, which includes both spaces and operators:

\begin{itemize}
\item We suppose that for each $T \in \calT$, and each $k \in \bbZ$ we are given a vector space $A^k(T)$. For $k<0$ we suppose $A^k(T) = 0$.

\item For every $T \in \calT$ and $k \in \bbZ$, we have an operator $\diff^k_T: A^k(T)  \to A^{k+1}(T)$ called \emph{differential}. Often we will denote it just as $\diff$. We require $\diff^{k+1}_T \circ \diff^k_T = 0$. This makes $A^\bs(T)$ into a complex.
\item  Given $T, T'$ in $\calT$ with $T' \subcell T$ we suppose we have \emph{restriction} maps:
\begin{equation}
\rest^k_{T'T} : A^k(T) \to A^k(T'),
\end{equation}
subject to the commutation relations (when $T'' \subcell T' \subcell T$):
\begin{itemize}
\item $\rest_{T'T} \diff^k_T = \diff^k_{T'} \rest_{T'T}$,
\item $\rest_{T''T}  = \rest_{T''T'} \rest_{T'T}$. 
\end{itemize}

\item For any $k$-dimensional cell $T$ in $\calT$ we suppose we have an \emph{evaluation map} $\eval : A^k(T) \to L(T)$. 
We suppose that the following formula holds, for $u \in A^{k-1}(T)$:
\begin{equation}\label{eq:Stokes}
\eval_T \diff_T u = \sum_{T' \in \partial T} \orient(T, T') \trans_{TT'} \eval_{T'} \rest_{T'T} u.
\end{equation}
 see also Remark \ref{rem:altstokes} below.
\end{itemize}
\end{definition}

\begin{remark}[presheaf interpretation]\label{rem:presh}
The first three points in the definition can be rephrased in terms of categories and sheaves. First, the cellular complex $\calT$ may be considered as a category, where the objects are the cells and the morphisms are the inclusion maps. Complexes of vector spaces also constitute a category. The first three points then say that the family $A^\bs(T)$, indexed by $T \in \calT$, equipped with restriction operators associated with inclusions maps, is a contravariant functor from $\calT$ to the category of complexes of vector spaces. It is therefore a  \emph{presheaf} of complexes (according to the definition in \cite{KasSch06} Chapter 17).
\end{remark}

\begin{remark}[inverse system interpretation]\label{rem:invsys} 
The cellular complex $\calT$ can also be interpreted as a partially ordered set, the order being the inclusion relation. Then the family $A^\bs(T)$, indexed by $T \in \calT$, equipped with restriction operators, constitutes an \emph{inverse system} (also called a projective system) of complexes. See for instance \cite{KasSch06} \S 2.1.
\end{remark}

\begin{remark}[Stokes' theorem]
Identity (\ref{eq:Stokes}) is a generalization of Stokes' theorem. Indeed Stokes' theorem may be regarded as the special case where the spaces $A^k(T)$ consist of smooth enough realvalued $k$-forms on $T$, the discrete vectorbundle is given simply by $L(T) = \bbR$ and $\trans_{TT'} = \id_\bbR$,  and the evaluation map $\eval_T : A^k(T) \to \bbR$ is integration of a $k$-form on a $k$-dimensional cell.
\end{remark}

\begin{definition}[gluing spaces]
If $\calT'$ is a cellular subcomplex of $\calT$, we define:
\begin{align}\label{eq:invlim}
A^k(\calT') = \{ & (u_T)_{T \in \calT'} \in \bigoplus_{T \in \calT'} A^k(T) \ : \ T' \subcell T \Rightarrow u_{T'} = \rest_{T'T} u_T \}.
\end{align}
\end{definition}
We notice that, in the special case where $T$ is a cell and $\subcells(T)$ denotes the cellular complex consisting of all the subcells of $T$ in $\calT$, then the restriction maps provide an isomorphism:
\begin{equation}
\rest: A^\bs(T) \to A^\bs(\subcells(T)).
\end{equation}
In what follows we will usually not distinguish between $A^\bs(T)$ and $A^\bs(\subcells(T))$.

\begin{remark}[continuity]
The condition that, for $T' \subcell T$ we have  $u_{T'} = \rest_{T'T} u_T$, for elements $u$ of $A^k(\calT')$,  can be interpreted as a continuity condition. Indeed when two cells have a common face, the condition enforces the two restrictions to the common face to be equal.
\end{remark}

\begin{remark}[glued spaces are inverse limits]
If $\calT'$ is a cellular subcomplex of $\calT$, the spaces $A^k(T)$ with $T\in \calT'$ constitute an inverse system, see Remark \ref{rem:invsys}. Then $A^k(\calT')$ is an \emph{inverse limit}, and this defines it up to unique isomorphism. See for instance \cite{KasSch06} Equation (2.1.2) p. 36. 
\end{remark}

The notation (\ref{eq:invlim}) will be used in particular in the following two cases:
\begin{itemize}
\item $\calT' = \calT$. The spaces of the form $A^k(\calT)$ are those typically used in a Galerkin method to solve a PDE on the set covered by $\calT$. Such spaces will be referred to as the global spaces, as opposed to the local spaces $A^k(T)$ for $T \in \calT$.
\item $\calT' = \partial T$ for a given $T \in \calT$. By that we mean the set of cells in $\calT$ included in the boundary of a given cell $T\in \calT$. That is, $\partial T$ denotes the cellular complex consisting of the strict subcells of $T$.
\end{itemize}

\begin{remark}[reformulation of Stokes' identity as a commutation relation]
\label{rem:altstokes}
Another way of formulating (\ref{eq:Stokes}) is that for any cellular subcomplex $\calT'$ of $\calT$, the evaluations $\eval_T : A^k(T) \to L(T)$  (for $T \in \calT^{\prime k}$) provide a morphism of complexes:
\begin{equation} \label{eq:evalmor}
\eval: A^\bs(\calT')\to \calC^\bs(\calT', L).
\end{equation}
We will later provide conditions under which the evaluation morphism (\ref{eq:evalmor}) induces isomorphisms on cohomology groups. This would be an analogue of de Rham's theorem which asserts that the de Rham map, from real valued differential forms to real valued cellular cochains, gives isomorphisms between the respective cohomologies. We therefore refer to the map in (\ref{eq:evalmor}) as the de Rham map.
\end{remark}

We denote by $A^k_0(T)$ the kernel of the induced map $\rest: A^k(T) \to A^k(\partial T)$. We consider that the boundary of a point is empty, so that if $T$ is a point $A^k_0(T) = A^k(T)$.

\begin{definition}[Flabby finite element systems]
We say that $A$ admits extensions on $T \in \calT$, if the restriction map induces a surjection:
\begin{equation}
\rest: A^\bs(T) \to A^k(\partial T) .
\end{equation}
We say that $A$ admit extensions on $\calT$ or is \emph{flabby}, if it admits extensions on each $T\in \calT$.
\end{definition}

This notion corresponds to that of \emph{flabby sheaves} (\emph{faisceaux flasques} in French \cite{God73}), due to the following result
\begin{proposition}
The FES $A$ admits extensions on $\calT$ if an only if, for any cellular complexes $\calT'', \cal T'$  such that $\calT'' \subseteq \calT' \subseteq \calT$,  the restriction  $A^\bs(\calT') \to A^\bs(\calT'')$ is onto.
\end{proposition}
\begin{proof} 
If one can extend from the boundary of a cell to the cell, then one can extend from subcomplexes to complexes, step by step, incrementing dimension by one each time.
\end{proof}

\begin{remark}[surjectivity of restrictions]
In particular if $A$ admits extensions, then, when $T'$ is a subcell of $T$, the restriction $A^\bs(T) \to A^\bs(T')$ is onto. However this is, in general, a strictly weaker condition than the extension property. To see this, consider for instance the finite element spaces $A^0(T)$ consisting of $\poly^1$ functions on a quadrilateral $S$, on its edges $E$ and on its vertices $V$. Then the restriction from $A^0(S)$ to each edge $A^0(E)$ is onto, as are the other restrictions from faces to subfaces, but the restriction from $A^0(S)$ to $A^0(\partial S)$ is not onto, since the latter has dimension 4 but the former had dimension only 3. In practice therefore, finite element spaces on a square therefore include, in addition to the affine functions, a bilinear function. 
\end{remark}

\begin{remark}[kernels of differentials] \label{rem:kjt} We define $K(T) = \rmH^0(A^\bs(T)) \approx \ker \rmd : A^0(T)$. We notice that we have induced maps $\rest_{T'T} : K(T) \to K(T')$, whenever $T' \subcell T$. We also notice that we have a welldefined map $j_T: K(T) \approx \rmH^0(A^\bs(T)) \to \rmH^0( \calC^\bs(\calT, L)) \approx L(T)$: starting with an element in $K(T)$, restrict it to a vertex $T'$ to get an element of $K(T')$, evaluate it to get an element of $L(T')$ and parallel transport it to get an element of $L(T)$ (the composition of these steps is independent of the choice of vertex and path from the vertex to $T$).
\end{remark}

\begin{definition}[Exactness of a FES]\mbox{}

\begin{itemize}
\item We say that $A^\bs$ is exact on a cell $T \in \calT$ when the following, equivalent, conditions hold:

\begin{itemize}
\item The following sequence is exact:
\begin{equation}\label{eq:exactd}
\xymatrix{
A^0(T) \ar[r]^\diff & A^1(T) \ar[r]^\diff & \ldots 
}
\end{equation}
and moreover the map $j_T: K(T) \to L(T)$ defined in Remark \ref{rem:kjt} is an isomorphism.

\item The de Rham map $A^\bs(T) \to \calC^\bs(T, L)$ induces isomorphisms on cohomology.

\end{itemize}

\item We say that $A^\bs$ is \emph{locally exact} on $\calT$ when $A^\bs$ is exact on each $T \in \calT$.
\end{itemize}
\end{definition}

\begin{proof}
The equivalence holds by Lemma \ref{lem:cohlt}.
\end{proof}

\begin{definition}\label{def:compat}
We say that $A$ is \emph{compatible} when it is flabby and is locally exact.
\end{definition}

\begin{remark} [The flat connections of upwinding and band gap computations] \label{rem:convdiff}

Some PDEs, such as convection diffusion equations, can be expressed with covariant derivatives \cite{Chr13FoCM}, see also \cite{WuXu18}.

Consider real valued differential forms. Choose $A\in \Omega^1(U)$, called the connection $1$-form. Define the covariant exterior derivative $\rmd_A$ by:
\begin{equation}
\rmd_A u = \rmd u + A \wedge u.
\end{equation}
Then we have:
\begin{equation}
\rmd_A \rmd_A u = (\rmd A) \wedge u.
\end{equation}
Then $\rmd A \in \Omega^2(U)$ is identified as the curvature of $A$. We suppose $\rmd A = 0$, so that the operators $\rmd_A$ constitute a complex. For any contractible subdomain $T$ of $U$ we may choose $\phi \in \Omega^0(T)$ such that $\rmd \phi = A$. Then we have, for $u \in \Omega^k(T)$:
\begin{equation}
\rmd_A u = \exp(-\phi) \rmd ( \exp (\phi) u ).
\end{equation}
Furthermore $\phi$ is uniquely determined, if we impose, in addition, the value of $\phi(x_T)$ for some point $x_T \in T$.

For every cell $T \in \calT$, choose an interior point $x_T \in T$. Let $\phi_T : T \to \bbR$ be the unique function such that $\phi_T(x_T) = 0$ and $\rmd \phi_T = A$. For $u \in \Omega^{k-1}(T)$ with $k = \dim T$, we notice:
\begin{align}
\int_T \exp(\phi_T) \rmd_A u & = \int_T \rmd \exp(\phi_T) u,\\
&= \sum_{T' \subcell T} \orient(T,T') \int_{T'} \exp(\phi_T) u|_T',\\
&= \sum_{T' \subcell T} \orient(T,T') \exp(\phi_{T}(x_{T'})) \int_{T'} \exp(\phi_{T'}) u|_T'.
\end{align}
When $T$ is $k$-dimensional, we let $\eval_T : \Omega^k(T) \to \bbR$ be the map:
\begin{equation}  
\eval_T : u \mapsto \int_T \exp(\phi_T) u.
\end{equation}
And, when $T'$ is a codimension $1$ face of $T$ we put:
\begin{equation}
\trans_{TT'} = \exp(\phi_T(x_{T'})).
\end{equation}
Then the preceding identity becomes:
\begin{equation}
\eval_T (\rmd_A u) = \sum_{T' \subcell T} \orient(T,T') \trans_{TT'} \eval_{T'} u|_{T'},
\end{equation}
as required in (\ref{eq:Stokes}).

As a slight variant, consider complex-valued differential forms (but $A$ still real valued), and let:
\begin{equation}
\rmd_A u = \rmd u + i A \wedge u,
\end{equation}
with comparable consequences. One can even restrict attention to $A$ constant. On a torus, $A$ is not in general globally of the form $\rmd \phi$, but of course locally this still holds and can be used in particular on individual cells. Such observations were made in \cite{DobPas01}\cite{BofConGas06} in the context of band gap computations for photonic cristals. The corresponding numerical methods are variants of exponential fitting.
\end{remark}

\subsection{de Rham type theorems.}

The following theorem extends Proposition 5.16 in \cite{ChrMunOwr11}:
\begin{theorem}\label{theo:derham}
Suppose that the element system $A$ is compatible. Then the evaluation maps $\eval: A^\bs(\calT) \to \calC^\bs(\calT, L)$ induces isomorphisms on cohomology groups.
\end{theorem}
\begin{proof}
Exactness gives that the map $A^\bs(T) \to \calC^\bs(\calS(T), L)$ induces isomorphims on cohomology groups.

From there the proof proceeds as in \cite{ChrMunOwr11}.
\end{proof}  

We also have the following extension of Proposition 5.17 in \cite{ChrMunOwr11}:
\begin{theorem} \label{theo:altcomp} Suppose that $A$ has extensions. Then $A$ is compatible if and only if the following condition holds:

For each $T \in \calT$ the sequence $A^\bs_0(T)$ has nontrivial cohomology only at index $k = \dim T$, and there the induced map:
\begin{equation}
\eval: \rmH^k A^\bs_0(T) \to L(T),
\end{equation}
is an isomorphism (it is well defined by (\ref{eq:Stokes})).

\end{theorem}

\begin{proof}
\tpoint For cellular complexes consisting only of vertices, the equivalence trivially holds because, when $T$ is a point, $A^\bs_0(T) = A^\bs(T)$.

\tpoint We suppose now that $m >0$ and that the equivalence has been proved for cellular complexes consisting of cells of dimension at most $n< m$. Consider a cellular complex consisting of cells of dimension at most $m$.

Let $T\in \calT$ be a cell of dimension $m$. We suppose that the finite element system $A$ is compatible on the boundary of $T$. Since the boundary is $(m-1)$-dimensional we may apply the de Rham theorem \ref{theo:derham} there. In other words $A^\bs(\partial T) \to \calC^\bs(\partial T,L)$ induces isomorphisms on cohomology.

We write the following diagram:
\begin{equation}\label{eq:shortexact}
\xymatrix{
0 \ar[r] & A^\bs_0(T) \ar[r] \ar[d]^\eval & A^\bs(T) \ar[r] \ar[d]^\eval & A^\bs(\partial T) \ar[r] \ar[d]^\eval  & 0 \\
0 \ar[r] & \calC^\bs_0(\subcells (T),L) \ar[r] & \calC^\bs(\subcells (T),L) \ar[r] & \calC^\bs(\partial T,L) \ar[r]  & 0 
}
\end{equation}
Comments:
\begin{itemize}
\item The complex $\calC^\bs_0(\subcells (T),L)$ is rather trivial. The terms consist of cochains that are zero on the boundary of $T$. In other words the only non-zero space in the complex is at index $k = \dim T$, where it is $L(T)$. 

\item On the rows, the second map is inclusion and the third arrow restriction. Both rows are short exact sequences of complexes. 

\item The vertical maps are the de Rham map. 

\item The diagram commutes.
\end{itemize}

We write the two long exact sequences corresponding to the two rows, and connect them by the map induced by the de Rham map.
\begin{equation}
\xymatrix{
\rmH^{k-1} A^\bs(T)  \ar[r] \ar[d]  & \rmH^{k-1} A^\bs(\partial T)  \ar[r] \ar[d] &\rmH^{k}  A^\bs_0(T) \ar[r] \ar[d] & \rmH^{k} A^\bs(T)  \ar[r] \ar[d]  & \rmH^{k} A^\bs(\partial T) \ar[d]\\
\rmH^{k-1} \calC^\bs(\subcells (T),L)  \ar[r]  & \rmH^{k-1} \calC^\bs(\partial T,L)  \ar[r] &\rmH^{k}  \calC^\bs_0(\subcells (T),L) \ar[r]  & \rmH^{k} \calC^\bs(\subcells (T),L)  \ar[r]  & \rmH^{k} \calC^\bs(\partial T,L) 
}
\end{equation}

The equivalence is now proved in two steps:
\begin{itemize}
\item Suppose that (\ref{eq:exactd}) is exact. Then the first and fourth vertical maps are isomorphisms. By the induction hypothesis the second and fifth are isomorphisms. By the five lemma, the third one is an isomorphism. This can be stated by the condition formulated in the theorem.

\item Suppose that the stated condition holds. One applies again the five lemma to the long exact sequence, and obtains now that $A^\bs(T)$ has the same cohomology as $\calC^\bs(\calS(T), L)$.
\end{itemize}
 
\end{proof}

\subsection{Extensions, dimension counts and harmonic interpolation.}
The following proposition almost exactly reproduces Proposition 2.2 in \cite{ChrRap16}.

\begin{proposition}\label{prop:extrec} Suppose that $A$ is an element system and that $T \in \calT$. We are interested only in a fixed index $k \in \bbN$. Suppose that, for each cell $U \in \partial T$, each element $v$ of $A^k_0(U)$ can be extended to an element $u $ of $A^k(T)$ in such a way that, $\rest_{UT} u = v$ and for each cell $U' \in \partial T$ with the same dimension as $U$, but different from $U$, we have $\rest_{U'T} u  = 0$. Then $A^k$ admits extensions on $T$.
\end{proposition}
\begin{proof}
In the situation described in the proposition we denote by $\ext_U v = u$ a chosen extension of $v$ (from $U$ to $T$).

Pick $v \in A^k(\partial T)$.  Define $u_{-1}= 0 \in A^k(T)$.

Pick $l \geq -1$ and suppose that we have a $u_{l} \in A^k(T)$ such that $v$  and $u_l$ have the same restrictions on all $l$-dimensional cells in $\partial T$. Put $w_l = v - \rest_{\partial T\, T} u_l \in A^k(\partial T)$. For each $(l+1)$-dimensional cell $U$ in $\partial T$, remark that $\rest_{U\partial T} w_l \in A^k_0(U)$, so we may extend it to the element $\ext_U \rest_{U\partial T} w_l \in A^k(T)$.
Then put:
\begin{equation}
u_{l+1} = u_l + \sum_{U \, : \, \dim U = l+1} \ext_U \rest_{U\partial T} w_l.
\end{equation}
Then $v$ and $u_{l+1}$ have the same restrictions on all $(l+1)$-dimensional cells in $\partial T$. 

We may repeat until $l+1 = \dim T$ and then $u_{l+1}$ is the required extension of $v$.
\end{proof}

\begin{proposition}[flabbyness vs dimension counts]
\label{prop:extdim}
Let $A$ be a FES on a cellular complex $\calT$. Then: 
\begin{itemize}
\item We have:
\begin{equation}\label{eq:bounddim}
\dim A^k (\calT) \leq \sum_{ T\in \calT}\dim A^k_0(T).
\end{equation}
\item Equality holds in (\ref{eq:bounddim}) if and only if $A^k$ admits extensions on each $T \in \calT$.
\end{itemize}
\end{proposition}
\begin{proof}
The proof in \cite{ChrRap16} works verbatim.
\end{proof}

\begin{definition}
Given a FES $A$ on a cellular complex $\calT$, a \emph{system of degrees of freedom} is a choice of subspace $F^k(T) \subseteq A^k(T)^\star$, for each $k \in \bbN$ and $T \in \calT$. In that situation we can define maps $\Phi^k(T) : A^k(T) \to \bigoplus_{T' \subcell T} F^k(T)^\star$ by, for $u \in A^k(T)$:
\begin{equation}
\Phi^k(T)u = (\langle \cdot, \rest_{T'T} u \rangle )_{T' \subcell T} 
\end{equation}
where the brackets denote the canonical bilinear pairing $F^k(T) \times A^k(T) \to \bbR$. We say that the system $F$ is unisolvent on $A$ if $\Phi^k(T)$ is an isomorphism for each $T \in \calT$. 
\end{definition}

We will use the following result:
\begin{proposition}[unisolvence of degrees of freedom and flabbyness] 
\label{prop:suffunisolve}
Suppose that $A$ is a FES on a cellular complex $\calT$. Suppose that $F$ is a system of degrees of freedom for $A$. Suppose that for each $T \in \calT$, the canonical map $A^k_0(T) \to F^k(T)^\star$ is injective. Suppose that $T \in \calT$ is such that:
\begin{equation} \label{eq:dimageq}
\dim A^k(T)  \geq \sum_{T' \subcell T} \dim F^k(T).
\end{equation}
Then $F$ is unisolvent on $A$ on the cellular complex $\subcells (T)$, $A$ is flabby on $\subcells(T)$ and equality holds in (\ref{eq:dimageq}). 
\end{proposition}
\begin{proof}
See Propositions 2.1 and 2.5 in \cite{ChrRap16}.
\end{proof}

\begin{example}
For each cell $T \in \calT$, equip $A^k(T)$ with a continuous scalar product $\langle \cdot | \cdot \rangle$, typically a variant of the $\rmL^2$ product. We define a system of degrees of freedom $F^k(T) \subseteq A^k(T)^\star$ as follows. For each $k$, we consider the following space of linear forms on $A^k(T)$:
\begin{align} 
F^k(T) & = \{\langle \cdot |  v \rangle \ : \ v \in \diff A^{k-1}_0(T)\} \oplus \{ \langle \rmd \cdot | v \rangle \ : \ v \in \diff A^{k}_0(T)\} \label{eq:harmdof} \\
& \qquad \oplus \{ l \circ \eval_T \ : \ l \in L(T)^\star \}, \label{eq:mindof}
\end{align}
where the last space in the direct sum should be included only for $k = \dim T$.

We call these the \emph{harmonic degrees of freedom}. For compatible finite element systems these degrees of freedom are unisolvent.

If one considers that the linear forms in  $F^k(T)$ are defined on more general fields, these DoFs yield a commuting interpolator onto $A^\bs(\calT)$, which we call the \emph{harmonic interpolator}. For more on this topic see \S 2.4 of \cite{ChrRap16}, in  particular Proposition 2.8 of that paper.

\end{example}

\begin{remark}[minimal spaces and harmonicity]
If the degrees of freedom consisting only of the spaces $F^k(T) = \{ l \circ \eval_T \ : \ l \in L(T)^\star \}$ are unisolvent, then the FES is minimal, and the fields provide an analogue of Whitney forms. One can obtain such a FES inside any compatible FES by imposing the degrees of freedom (\ref{eq:harmdof}) to be zero, while (\ref{eq:mindof}) are kept free.
This generalizes the construction of Whitney forms on cellular complexes given in \cite{Chr08M3AS}.
\end{remark}

\subsection{Discrete vector bundles: a dual picture \label{par:dualvb}}

Notice that the degrees of freedom $l \circ \eval_T$ for  $l \in L(T)^\star$ appearing in (\ref{eq:harmdof}) play a special role. In practice they often appear in a slightly different way, namely as integration against certain fields, forming a space $M(T)$ which is more tangible than $L(T)$ (the parallel transport operators acting on $M(T)$ can be more natural for instance).

We now make some remarks on this alternative point of view.

We suppose that we have for each $T$ a vector space $M(T)$ and a bilinear form $\langle \cdot , \cdot \rangle_T$ on $A(T) \times M(T)$. Moreover, when $T'$ has codimension $1$ in $T$ we suppose we have a bijective (linear) restriction map $s_{T'T}: M(T) \to M(T')$, subject to the condition that $s_{TT_0'}s_{T_0'T''} =  s_{TT_1'}s_{T_1'T''}$. The generalized Stokes theorem takes the form, for $\phi \in M(T)$ and $u \in A^{k-1}(T)$:
\begin{equation}\label{eq:ibpma}
\langle \rmd u, \phi \rangle_T = \sum_{T'} \orient(T, T') \langle \rest_{T'T} u,  s_{T'T} \phi \rangle_{T'}.
\end{equation}

In practice, this formula often arises as follows. The bilinear form $\langle \cdot , \cdot \rangle_T$ on $A(T) \times M(T)$ is the $\rmL^2$ scalar product (with respect to, say, the standard Euclidean metric). The space $M(T)$ is the kernel of the formal adjoint of $\rmd$. Identity (\ref{eq:ibpma}) is obtained by integration by parts, $m$ times when $\rmd$ is a differential operator of order $m$. Only boundary terms remain, by definition of the kernel of the formal adjoint. For boundary cells $T'$, the space $M(T')$ could be obtained as the kernel of a formal adjoint, or could appear as natural restrictions to $T'$ of elements of $M(T)$.  

For instance, one integrates the divergence of a vector field against the constants, which constitute the kernel of the gradient. Then the restriction operator on the constants is the usual trace and the restriction operator on the vector field is the trace of the normal component.

We will be interested in more complicated examples, for instance, integration of vector fields against the kernel of the deformation operator (i.e. rigid motions) which is the formal adjoint of the divergence operator acting on symmetric matrices), and integration of functions against the kernel of the Airy operator (i.e. affine functions), which is the formal adjoint of the Saint-Venant compatibility operator.

This data provides the discrete vector bundle defined by $L(T) = M(T)^\star $ and $\trans_{TT'} = s_{T'T}^\star$. Moreover it gives evaluation maps $\eval_T: A^k(T) \to L(T)$ defined by $u \mapsto \langle u, \cdot \rangle_T$.

Conversely, given the data $L$, $\trans$ and $\eval$,  one can define $M(T) = L(T)^\star$, $\langle u, \phi \rangle_T = \phi(\eval_T(u))$ for $\phi \in M(T)$ and $u \in A^k(T)$ and $s_{T'T} = \trans_{TT'}^\star$, so the two points of view are equivalent.

\section{Elasticity \label{sec:elasticity}}

\subsection{Spaces}

We work in dimension $2$.
\begin{itemize}
\item We denote by $\bbV = \bbR^2$ the space of column vectors and by $\bbV^\transp$ the space of row vectors,
\item We denote by $\bbM= \bbR^{2 \times 2}$ the space of matrices.
\item We denote by $\bbS= \bbR^{2 \times 2}_\sym$ the space of symmetric $2\times 2$ matrices and by $\bbK=\bbR^{2 \times 2}_{\Skew}$ the space of skewsymmetric  $2 \times 2$-matrices.
\end{itemize}

If $T$ is a subset of $\bbR^2$ (typically a simplex of dimension $0$, $1$ or $2$, or an open subset of $\bbR^2$), and $\bbE$ is a vector space (such as one of $\bbR, \bbV, \bbV^\transp, \bbM, \bbS, \bbK$, or a product of such spaces) we denote by $\Gamma(T, \bbE)$ the set of maps from $T$ to $\bbE$.  

These maps should be smooth enough that differential operators and traces (in the sense of restrictions to subsets of the boundary) make sense, so that one can define a suitable finite element system as subspaces. To get flabbyness, the regularity of the spaces should not be too big. Exactness under the differentials is also a regularity dependent issue. We will get back to the problem of getting the regularity of the fields right, to accomodate these constraints.

\subsection{Differential operators}
We will use the following differential operators, on scalar, vector and matrix fields, defined on some connected domain in $\bbR^2$.

\begin{itemize}
\item The gradient and curl of a scalar field are (row) vectorfields defined by:
\begin{align}
\grad [u] & = [\partial_1 u, \partial_2 u],\\
\curl [u] & = [\partial_2 u , - \partial_1 u].
\end{align}

\item The divergence and curl of (row) vectorfields are defined by:
\begin{align}
\div [u_1, u_2] = \partial_1 u_1 + \partial_2 u_2,\\
\curl [u_1, u_2] = \partial_1 u_ 2 -\partial_2 u_1.
\end{align}

\item The gradient of a (column) vectorfield is defined by:
\begin{equation}
\grad \left[\begin{array}{r}
u_1\\
u_2
\end{array} \right]
= \left[\begin{array}{cc}
\partial_1 u_1 &  \partial_2 u_1\\
\partial_1 u_2 & \partial_2 u_2
\end{array}\right]
\end{equation}

\item The deformation of a (column) vectorfield is the symmetric part of its gradient:
\begin{equation}
 \defo \left[\begin{array}{r}
u_1\\
u_2
\end{array} \right]
= \left[\begin{array}{cc}
\partial_1 u_1 &  1/2(\partial_2 u_1 + \partial_1 u_2)\\
1/2(\partial_1 u_2 + \partial_2 u_1) & \partial_2 u_2
\end{array}\right]
\end{equation}

\item The divergence of a matrix field is the (column) vectorfield defined by:
\begin{equation}
\div  \left[\begin{array}{rr}
u_{11} & u_{12} \\
u_{21} & u_{22}
\end{array} \right] 
= \left[\begin{array}{r}
\partial_1 u_{11} + \partial_2 u_{12}\\
\partial_1 u_{21} + \partial_2 u_{22}
\end{array}\right].
\end{equation}

\item The Airy operator acts as follows on a scalar field $u$:
\begin{equation}
\airy [u] = \left[ \begin{array}{rr}
\partial_{22} u & -\partial_{21} u \\
-\partial_{12} u & \partial_{11} u
\end{array} \right]
\end{equation}

\item The Saint Venant operator is the formal adjoint of the Airy operator:
\begin{equation}
\sven  \left[\begin{array}{rr}
u_{11} & u_{12} \\
u_{21} & u_{22}
\end{array} \right] 
= \partial_{22} u_{11} + \partial_{11} u_{22} - \partial_{12} u_{21} - \partial_{21} u_{12}.
\end{equation}
It encodes the Saint-Venant compatibility conditions for being a deformation tensor.
The sven operator may also by interpreted as the linearization of scalar curvature around the standard metric.
\end{itemize}

\begin{remark}[Hessian]
The Hessian of a scalar field $u$ is:
\begin{equation}
\hess(u) = \left[ \begin{array}{rr}
\partial_{11} u & \partial_{12} u \\
\partial_{21} u & \partial_{22} u
\end{array} \right]
\end{equation}

\begin{itemize}
\item We define the matrix $J$ as:
\begin{equation}\label{eq:jdef}
J = \left[ \begin{array}{ll}
0 & -1 \\
1  & 0
\end{array} \right]
\end{equation}
It encodes a direct rotation around the origin by $\pi/2$. Then the Hessian and the Airy operator are linked by:
\begin{equation}
\airy (u) = J^\transp \hess(u) J.
\end{equation}

\item We can also introduce the operator on matrices:
\begin{equation}\label{eq:kdef}
K u = u^\transp - \tr(u)I_2.
\end{equation}
Then the Hessian and the Airy operator are related by:
\begin{equation}
\airy(u) = - K \hess(u).
\end{equation}
\end{itemize}
\end{remark}

\begin{remark}[Formulas in different bases] 
We will often work with an additional orthonormal oriented basis $(\tau, \nu)$ of $\bbV$. Well-known formulas will be used without further ado. For instance, for a scalar field $u$, $\curl u = \partial_\nu u \tau - \partial_\tau u \nu$. For a column vector field one would write $\curl u = \partial_\tau u \cdot \nu - \partial_\nu u \cdot \tau$. Notice however that we have chosen to use differential operators only row wise. For a row vector field $u$ the formula thus becomes $\curl u = \partial_\tau u \nu - \partial_\nu u \tau$, given that $\tau$ and $\nu$ are column vectors.    
\end{remark}

\begin{remark}[Differential operators act on rows]
We let all differential operators act row wise. To account for differential operators acting column wise, we combine with transposition, denoted $\transp : u \mapsto u^\transp$.
\end{remark}

\begin{remark} We let the $\curl$ operator act on matrices row wise. Then we have the formula:
\begin{equation}
\curl u = - J \div (K u),
\end{equation}
where $J$ and $K$ were introduced in (\ref{eq:jdef}) and (\ref{eq:kdef}). This can be used to prove invariance properties of the curl of matrices under orthogonal basis change.
\end{remark}

\subsection{Elasticity complexes and diagram chasing}

The preceding operators may be arranged into sequences as follows:
\begin{proposition} [Elasticity stress complex]
We have a sequence:
\begin{equation}\label{eq:elastress}
\xymatrix{
\rmH^2(U, \bbR) \ar[r]^\airy & \rmH^0_{\div}(U,\bbS) \ar[r]^\div & \rmH^0(U, \bbV).
}
\end{equation}
Here:
\begin{equation}
\rmH^0_{\div}(U,\bbS) = \{u \in \rmH^0(U, \bbS) \ : \ \div u \in \rmH^0(U, \bbV) \}.
\end{equation}

The kernel of the Airy operator consists of the affine functions. 
The divergence operator is surjective.
If $U$ is contractible the sequence is exact.
\end{proposition}

\begin{proposition} [Elasticity strain complex]
We have a sequence:
\begin{equation}\label{eq:elastrain}
\xymatrix{
\rmH^2(U, \bbV) \ar[r]^\defo & \rmH^1_{\sven}(U,\bbS) \ar[r]^\sven & \rmH^0(U, \bbR).
}
\end{equation}
Here:
\begin{equation}
\rmH^1_{\sven}(U,\bbS) = \{u \in \rmH^1(U, \bbS) \ : \ \sven u \in \rmH^0(U, \bbR) \}. \label{eq:shighreg}
\end{equation}

The kernel of the deformation operator consists of the rigid motions.
The sven operator is surjective.
If $U$ is contractible the sequence is exact.
\end{proposition}

These complexes can be deduced from vector valued de Rham sequences through a diagram chase of the following type:
\begin{proposition}\label{prop:chase}
Suppose we have two complexes linked in a commuting diagram
\begin{equation}
\xymatrix{
X^0 \ar[r]^\rmd                 & X^1 \ar[r]^\rmd                 & \ldots \ar[r]^\rmd  & X^i \ar[r]^\rmd & X^{i+1} \ar[r]^{\rmd} & \ldots ,\\
Y^0 \ar[r]^\rmd  \ar[ur]^{\phi^0} & Y^1 \ar[r]^\rmd  \ar[ur]^{\phi^1} &  \ldots \ar[r]^\rmd  \ar[ur]^{\phi^{i-1}} &Y^i \ar[r]^\rmd  \ar[ur]^{\phi^i} & Y^{i+1} \ar[r]^\rmd \ar[ur]^{\phi^{i+1}}  & \ldots 
}
\end{equation}
Suppose furthermore that we have an index $i$ such that:
\begin{itemize} 
\item for $j <i$, $\phi^j$ is injective,
\item $\phi^i$ is bijective,
\item for $j >i$, $\phi^j$ is surjective. 
\end{itemize}
Then we get a sequence, with $\bbd = \rmd \Psi \rmd$ and $\Psi = (\phi^i)^{-1}$:
\begin{equation}
\xymatrix{
\ldots \ar[r]^\rmd & \coker \phi^i \ar[r]^\rmd & \coker \phi^{i-1} \ar[r]^\bbd & \ker \phi^{i+1} \ar[r]^{\rmd} &  \ker \phi^{i+2} \ar[r]^{\rmd} & \ldots
}
\end{equation}
Furthermore, if $X^\bs$ and $Y^\bs$ are exact then this new sequence is also exact.
\end{proposition}

\begin{proposition} We  have the commuting diagram:
\begin{equation} \label{eq:chasestress}
\xymatrix{
\rmH^2(U, \bbR) \ar[r]^\curl & \rmH^1(U,\bbV^\transp) \ar[r]^\div & \rmH^0(U, \bbR),\\
\rmH^1(U, \bbV) \ar[r]^\curl \ar[ru]^\transp & \rmH^0_{\div}(U,\bbM) \ar[r]^\div \ar[ru]^\Skew & \rmH^0(U, \bbV).
}
\end{equation}
Here $\transp$ denotes transposition and:
\begin{equation}
\Skew \left[ \begin{array}{rr}
u_{11} & u_{12} \\
u_{21} & u_{22}
\end{array} \right] = u_{21} - u_{12}.
\end{equation}
The elasticity stress complex (\ref{eq:elastress}) comes from (\ref{eq:chasestress}) by Proposition \ref{prop:chase}.
\end{proposition}

\begin{proposition} We have a commuting diagram:
\begin{equation} \label{eq:chasestrain}
\xymatrix{
\rmH^2(U, \bbV) \ar[r]^\grad & \rmH^1_{\sven}(U,\bbM) \ar[r]^\curl & \rmH^0_{\curl \transp}(U, \bbV),\\
\rmH^1(U, \bbR) \ar[r]^\grad \ar[ru]^\Skew & \rmH^0_{\curl}(U,\bbV^\transp) \ar[r]^\curl \ar[ru]^\transp & \rmH^0(U, \bbR).
}
\end{equation}
Here:
\begin{equation}
\rmH^1_{\sven}(U,\bbM) = \{u \in \rmH^1(U, \bbM) \ : \ \sven u \in \rmH^0(U, \bbR) \}.
\end{equation}
Moreover $\transp$ denotes transposition and:
\begin{equation}
\Skew [u] =  \left[ \begin{array}{rr}
0 & u \\
- u & 0
\end{array} \right].
\end{equation}
The elasticity strain complex (\ref{eq:elastrain}) comes from (\ref{eq:chasestrain}) by Proposition \ref{prop:chase}.
\end{proposition}

\begin{remark}[a variant with modified regularity]
\label{rem:chasestressbis}
On the bottom row of diagram (\ref{eq:chasestress}) the differential operators act row-wise and the Sobolev spaces are defined row-wise. The middle space in the bottom row is not stable under taking the transpose. One could achieve that by enforcing conditions on the vertical divergence throughout the complex:

\begin{equation} \label{eq:chasestressbis}
\xymatrix{
\rmH^2(U, \bbR) \ar[r]^\curl & \rmH^1_{\div}(U,\bbV^\transp) \ar[r]^\div & \rmH^1(U, \bbR),\\
\rmH^1_{\div \transp}(U, \bbV) \ar[r]^\curl \ar[ru]^\transp & \rmH^0_{\tiny
\div, \div \transp}(U,\bbM) \ar[r]^\div \ar[ru]^\Skew & \rmH^0(U, \bbV).
}
\end{equation}
Here:
\begin{align}
\rmH^1_{\div \transp}(U, \bbV) & = \{ u \in \rmH^1(U, \bbV) \ : \ \div u^\transp \in \rmH^1(U, \bbR) \},\\
\rmH^0_{\tiny \div, \div \transp}( U, \bbM) & = \{u \in \rmH^0( U, \bbM) \ : \ \div u \in \rmH^0(U, \bbV) \wedge  \div u^\transp \in \rmH^0(U, \bbV) \}.
\end{align}
The diagram chase applied to (\ref{eq:chasestressbis}) yields the same elasticity complex as (\ref{eq:chasestress}), namely (\ref{eq:elastress}). 
\end{remark}

\begin{remark}[regularity couples rows]
On the top-row of diagram (\ref{eq:chasestrain}) the differential operators act row-wise, but the Sobolev spaces have a regularity whose definition couples the rows. The middle space is stable under taking the transpose.
\end{remark}

\begin{remark}[Boundary conditions in the strain complex] Here we lower the regularity throughout the complex by 1, so that we consider in particular the space:
\begin{equation}
\rmH^0_{\sven}(U,\bbS) = \{ u \in \rmH^0(U, \bbS) \ : \ \sven u \in \rmH^{-1}(U, \bbR) \}.
\end{equation}
We are interested in finding the natural boundary conditions for this space. For $u$ in this space the $\tau\tau$ component on the boundary makes sense. Indeed, there is $u' \in \rmH^1(U, \bbS)$ such that $\sven u' = \sven u$ and $v \in \rmH^1(U, \bbV)$ such that $\defo v = u - u'$ (on noncontractible domains use the regular decomposition). Now $u'$ has well defined traces on the boundary, in particular $\tau \tau$ traces, say in $\rmL^2(\partial U)$. Moreover:
\begin{equation}
((\defo v) \tau) \cdot \tau = ((\grad v) \tau) \cdot \tau,
\end{equation}
and the trace of $(\grad v) \tau$ is well defined in $\rmH^{-1/2}(\partial U)$. But then, taking scalar product with $\tau$ is problematic. The components of $\tau$ are not in $\rmH^{1/2}(\partial U)$ on non-smooth domains, due to possible discontinuities. On a triangle one can still conclude that the trace of $((\defo v) \tau) \cdot \tau$ is well defined in the dual of $\rmH^{1/2}_{00}$ of each face. But it is problematic that this not enough to define $\tau \tau$ traces in, say, the dual of Lipschitz functions on the whole boundary. 

Characterizing the space of traces seems simpler for smooth boundaries. The difficulties with triangles can then be expressed as being related to the presence of infinite extrinsic curvature at vertices.

One should also prove that an element of $\rmH^0_{\sven}(U,\bbS)$ has zero $\tau \tau$ traces iff it is the limit of elements of $\rmC^\infty_c(U, \bbS)$.
\end{remark}

\begin{remark}[Strain complex with lower regularity]\label{rem:straincomplexlow}

As an alternative to the elasticity strain complex (\ref{eq:elastrain}) we can consider:
\begin{equation}\label{eq:elastrainlow}
\xymatrix{
\rmH^1_{\curl \transp} (U, \bbV) \ar[r]^\defo & \rmH^0_{ \tiny \curl, \sven}(U,\bbS) \ar[r]^\sven & \rmH^0(U, \bbR).
}
\end{equation}
Here:
\begin{align}
\rmH^1_{\curl \transp} (U, \bbV) & = \{ u \in \rmH^1(U, \bbV) \ : \ \curl u^\transp \in \rmH^1(U) \},\\
\rmH^0_{\tiny \curl, \sven}(U,\bbS) & = \{u \in \rmH^0(U, \bbS) \ : \ \curl u \in \rmH^0(U, \bbV) \wedge \sven u \in \rmH^0(U, \bbR) \}. \label{eq:slowreg}
\end{align}
This complex can be deduced from the following diagram:
\begin{equation} \label{eq:chasestrainlow}
\xymatrix{
\rmH^1_{\curl \transp} (U, \bbV) \ar[r]^\grad & \rmH^0_{\tiny \begin{matrix} \curl,  \curl \transp \\ \sven \end{matrix}}(U,\bbM) \ar[r]^\curl & \rmH^0_{\curl \transp}(U, \bbV),\\
\rmH^1(U, \bbR) \ar[r]^\grad \ar[ru]^\Skew & \rmH^0_{\curl}(U,\bbV^\transp) \ar[r]^\curl \ar[ru]^\transp & \rmH^0(U, \bbR).
}
\end{equation}
Here:
\begin{align}
\rmH^0_{\tiny \begin{matrix} \curl,  \curl \transp \\ \sven \end{matrix}} (U,\bbM)   = &  \{u \in \rmH^0(U, \bbM) \ : \curl u \in \rmH^0(U, \bbV) \wedge \\ 
&  \curl u^\transp \in \rmH^0(U, \bbV) \wedge \sven u \in \rmH^0(U, \bbR) \}.
\end{align}
One may check that the $\Skew$ operator on $\rmH^1(U, \bbR)$ maps \emph{onto} the antisymmetric elements of this space.
\end{remark}

\begin{remark}[regularity and partitions of unity]\label{rem:sobchoice}
The spaces defined in (\ref{eq:shighreg}) and (\ref{eq:slowreg}) are stable under multiplication by smooth functions. Thus, partition of unity techniques apply to them, enabling one to glue together smooth enough fields. The regularity in (\ref{eq:shighreg}) and (\ref{eq:slowreg}) seems quite adapted to FES techniques, which is about gluing together finite element fields. 

Notice that, on the other hand, the $\rmL^2$ based graph norm space:
\begin{equation}
 \{u \in \rmH^0(U, \bbS) \ : \ \sven u \in \rmH^0(U, \bbR) \},
\end{equation}
is not stable under multiplication by smooth functions. It seems less amenable to FES techniques.

We see that we have several choices of Sobolev norms in the strain complex. In the stress complex (\ref{eq:elastress}), on the other hand, the situation is simpler. The last two spaces are dictated by the Hellinger-Reissner principle, and this fixes the first space.

In dimension 3, this suggests the following two choices of norms, in the elasticity complex. For higher regularity:
\begin{equation}\label{eq:3delasticity}
\xymatrixcolsep{1.5cm}
\xymatrix{
\rmH^2(U, \bbV) \ar[r]^-\defo & \calX^1 \ar[r]^-\curltcurl & \rmH^0_{\div}(U,\bbS)  \ar[r]^-\div & \rmH^0(U, \bbV),
}
\end{equation}
with:
\begin{equation}
\calX^1 = \{u \in \rmH^1(U, \bbS) \ : \ \curltcurl u \in \rmH^{0}(U, \bbS)\}. 
\end{equation}
For lower regularity:
\begin{equation}
\xymatrixcolsep{1.5cm}
\xymatrix{
\rmH^1_{\curl \transp}(U, \bbV) \ar[r]^-\defo & \calY^1 \ar[r]^-\curltcurl & \rmH^0_{\div}(U,\bbS)  \ar[r]^-\div & \rmH^0(U, \bbV).
}
\end{equation}
with:
\begin{equation}
\calY^1 = \{u \in \rmH^0(U, \bbS) \ : \ \curl u \in \rmH^0(U, \bbM) \wedge \curltcurl u \in \rmH^{0}(U, \bbS)\}. 
\end{equation}
On the other hand, the complex:
\begin{equation}
\xymatrixcolsep{1.5cm}
\xymatrix{
\rmH^1(U, \bbV) \ar[r]^-\defo & \calZ^1 \ar[r]^-\curltcurl & \rmH^0_{\div}(U,\bbS)  \ar[r]^-\div & \rmH^0(U, \bbV),
}
\end{equation}
constructed on the $\rmL^2$ based graph norm space:
\begin{equation}
\calZ^1= \{u \in \rmH^0(U, \bbS) \ : \ \curltcurl u \in \rmH^0(U, \bbS) \},
\end{equation}
is avoided, as it is not stable under multiplication by smooth functions.

For Regge calculus, the functional framework of \cite{Chr11NM} was based on:
\begin{equation}
 \{u \in \rmH^0(U, \bbS) \ : \ \curltcurl u \in \rmH^{-1}(U, \bbS) \},
\end{equation}
which corresponds to lowering the regularity by $1$ throughout the complex (\ref{eq:3delasticity}).

Exactness properties of all these sequences can be deduced from results in \cite{ArnHu19}. 
\end{remark}

\subsection{Poincar\'e operators}

Poincar\'e and Koszul operators are central to the construction of the finite element de Rham complexes of \cite{Ned80}\cite{Hip99}\cite{ArnFalWin06}. They were also essential to the construction of the smooth finite element de Rham complexes in \cite{ChrHu18}. In \cite{ChrHuSan18} we provided similar operators for elasticity complexes. We detail them here, for the two different complexes we have in dimension 2. They will also be used for the construction of finite element elasticity complexes.

\paragraph{Stress elasticity complex.}

We suppose that the domains $U$ is starshaped with respect to the origin $0$. 

For $x=(x_1,x_2) \in \bbR^2$ we define $x^\rota =(-x_2,x_1)$. Both are identified with column vectors.

We define Poincar\'e operators for elasticity as follows.

\begin{definition}
We define $\poincare_1 : \Gamma(U,\bbS) \to \Gamma(U, \bbR)$ by:
\begin{equation}
(\poincare_1 u)(x)  = \int_0^1 (1-t) x^{\rota \transp} u(tx)x^\rota \rmd t.
\end{equation}

We define $\poincare_2 : \Gamma(U,\bbV) \to \Gamma(U, \bbS)$ by:
\begin{align}
(\poincare_2 u)(x) & = \sym (\int_0^1 t u(tx) x^\transp + \curl (\int_0^1 t(1-t) x^{\rota \transp} u(tx) x \rmd t)),
\end{align}
where the $\curl$ acts rowwise to produce a matrix.
\end{definition}

\begin{proposition}
We have the following identities. 
\begin{itemize}
\item Null-homotopy:
\begin{align}
\poincare_1 \airy u & = u - j(u),\\
\poincare_2 \div v + \airy \poincare_1 v & = v,\\
\div \poincare_2 w & = w.
\end{align}
where $j(u)$ is the affine function $x \mapsto  u(0) + \grad u (0) x$.

\item Sequence property:
\begin{equation}
\poincare_1 \poincare_2 = 0.
\end{equation}

\item These operators preserve polynomials, moreover:\\
-- $\poincare_1$ increases polynomial degree by $2$ (at most).\\
-- $\poincare_2$ increases polynomial degree by $1$ (at most).
\end{itemize}
\end{proposition}

With these we may define two koszul operators.

We may remark that on constant fields these operators reduce to:
\begin{align}
\poincare_1 u (x) & = 1/2 x^\rota u x^\rota,\\
\end{align}
and:
\begin{equation}
\poincare_2 u (x) =   2/3   \left[\begin{array}{cc}
x_1 u_1 &  1/2(x_2 u_1 + x_1 u_2)\\
1/2(x_1 u_2 + x_2 u_1) & x_2 u_2
\end{array}\right] 
\end{equation}

\paragraph{Strain elasticity complex.}

For the strain elasticity complex we also have Poincar\'e operators (we use the same notation). We suppose that the domain $U$ is starshaped with respect to the origin $0$. In practice we will use these operators with respect to other base points, namely the central vertex of the Clough-Tocher split.

\begin{definition}
We define $\poincare_1 : \Gamma(U, \bbS) \to \Gamma(U, \bbV)$ by:
\begin{equation}
\poincare_1 u = \int_{0}^{1}u_{tx}^\transp {x} \rmd t +\int_{0}^{1}(1-t)x^{\rota} (\curl u)^\transp_{tx} x \rmd t,
\end{equation}
and $\poincare_2 : \Gamma(U, \bbR) \to \Gamma(U, \bbS)$ by:
\begin{equation}
\poincare_{2}u = {x}^{\rota} \left (\int_{0}^{1}t(1-t)u_{tx} \rmd t \right ) {x}^{\rota \transp}.
\end{equation}
\end{definition}

\begin{proposition} The Poincar\'e operators have the following properties:
\begin{itemize}
\item Null-homotopy:
We have, for $u\in \Gamma(U, \bbV)$, $v \in \Gamma(U, \bbS)$ and $w \in \Gamma(U, \bbR)$:
\begin{align}
\poincare_{1} \defo u  & = u - j(u),\\
\poincare_{2} \sven v + \defo \poincare_{1} v & = v,\\
\sven \poincare_{2}  w & = w,
\end{align}
where $j(u)$ is the rigid motion $x \mapsto u(0) + 1/2 \curl \transp u (0) x^\rota$. 
\item Sequence property:
\begin{equation}
\poincare_1 \poincare_2 = 0.
\end{equation}

\item These operators preserve polynomials, moreover:\\
-- $\poincare_1$ increases polynomial degree by $1$ (at most).\\
-- $\poincare_2$ increases polynomial degree by $2$ (at most).
\end{itemize}
\end{proposition}

\begin{definition}[Koszul operators]
Taking the Poincar\'{e} operators on homogeneous polynomials of degree $r$, we obtain the Koszul operators. More generally, we let the operator $\koszul_{1}^{r}: \Gamma(U, \mathbb{S}) \mapsto \Gamma(U, \mathbb{V})$ be given by:
\begin{equation}
\koszul_{1}^{r}u = \frac{1}{r+1} u {x}+\frac{1}{(r+1)(r+2)}{x}^\rota (\curl u)^\transp {x},
\end{equation}
whereas the operator $\koszul_{2}^{r}: \Gamma(U, \bbR) \to \Gamma(U, \bbS)$ is given by:
\begin{equation}
\koszul_{2}^{r} u = \frac{1}{(r+2)(r+3)}{x}^{\rota}u  {x}^{\rota \transp}.
\end{equation}
\end{definition}

The next lemma shows some regularity of these operators.
\begin{lemma}\label{lem:kosreg}
For $w$ piecewise smooth on $\calR(T)$, we have $\poincare_{2} w \in \rmH^0_{\curl}(T, \bbS)$.
\end{lemma}
\begin{proof}
By straightforward calculations, $(x^{\rota} x^{\rota \transp}) x=0$ and $ \transp \curl (x^{\rota} x^{\rota \transp}) x=0$. This implies that $wx^{\rota \transp} x^{\rota}$ and $w \curl(x^{\rota \transp} x^{\rota})$ have continuous tangential components even if $w$ is discontinuous across of the interior edges of the Clough-Tocher split.
\end{proof}

\section{FES for the stress complex \label{sec:fesstress}}

\subsection{Induced operators and discrete vectorbundle}

We consider a vertex $V$, in an edge $E$, in a triangle $T$. The oriented unit tangent on $E$ is denoted $\tau$ and the normal is denoted $\nu$, so that $(\tau , \nu)$ is an oriented orthonormal basis of $\bbV$.

Spaces and operators for the stress complex may be arranged in the following commuting diagram:

\begin{tikzpicture}

\matrix (L) [matrix of math nodes, row sep = 2.5cm, column sep = 2cm] {
\Gamma(T, \bbR) & \Gamma(T, \bbS) & \Gamma(T, \bbV) \\
\Gamma(E, \bbR \times \bbR) & \Gamma(E, \bbR \times \bbR) \\
\Gamma(V, \bbR \times \bbV)\\
};

\path[->] 
(L-1-1) edge node[above] {$\airy$} (L-1-2)
(L-1-2) edge node[above] {$\div $} (L-1-3)
(L-2-1) edge (L-2-2)
(L-1-1) edge (L-2-1)
(L-2-1) edge (L-3-1)
(L-1-2) edge (L-2-2)
(L-1-1) edge [out= -125, in = 125]  node [left]{\footnotesize $\begin{array}{c}
u \\ \downarrow \\ (u, \transp \grad u) \end{array}$} (L-3-1)
;

\draw (L-1-2) + (0, -0.7) node [below right] {\footnotesize $\begin{array}{l}
\quad u \\ \quad \downarrow \\ (u\nu \cdot \tau, u\nu \cdot \nu) \end{array}$};

\draw (L-1-1) + (0, -0.7) node[below right] {\footnotesize $\begin{array}{l} 
\quad u\\ \quad \downarrow \\ (u, \partial_\nu u)\end{array}$};

\draw (L-2-1) + (0, -0.7) node[below right] {\footnotesize $\begin{array}{l}
(u,v)  \rightarrow  (- \partial_\tau v, \partial_\tau^2 u) \\
\quad \downarrow\\
 (u, \partial_\tau u \tau + v \nu) 
\end{array}$};

\end{tikzpicture}

\begin{remark}
With the preceding notations we have that:
\begin{itemize}
\item $K(T)$ consists of affine functions.
\item $K(E)$ consists of pairs $(u,v)$ of realvalued functions on $E$ such that $v$ is constant and $u$ is affine.
\item $K(V) = \bbR \times \bbV$.
\end{itemize}
\end{remark}

To complete the picture we need to define a discrete vector bundle. We take the dual point of view developed in \S \ref{par:dualvb} page \pageref{par:dualvb}.
 
\begin{itemize}
\item We let $M(T)$ denote the space of rigid motions. The rigid motions appear as the kernel of the formal adjoint of $\div : \Gamma(T, \bbS) \to \Gamma(T, \bbV)$, acting on symmetric matrices, namely the deformation operator $\defo$.

\item We let $M(E)$ be the kernel of the (formal) adjoint of $\diff^0_E : \Gamma(E, \bbR \times \bbR) \to \Gamma(E, \bbR \times \bbR)$, which is: 
\begin{equation}
(u,v) \mapsto (\partial^2_\tau v, \partial_{\tau} u).
\end{equation}
Therefore $M(E)$ consist of pairs $(u,v)$ of functions on $E$, where $u$ is constant and $v$ is affine.
 
\item We let $M(V) = \bbR \times \bbV$.
\end{itemize}

We define bijective restriction operators $M(T) \to M(E)$ and $M(E) \to M(V)$  as follows. 

\begin{itemize}
\item We define $M(T) \to M(E)$ as the map sending a rigid motion  $\phi$ to the pair $(u\cdot \tau, u \cdot \nu)$, where restriciton to $E$ is implied. Indeed the tangent component of $\phi$ on $E$ is constant and the normal component on $E$ is affine.

\item We define $M(E) \to M(V)$ as the map sending $(\phi,\psi) \in M(E)$ to the pair $(-\partial_\tau \psi, \psi \tau - \phi \nu)$, evaluated at $V$, which is in $M(V)$. 

\item There is one commuting diagram to check, when $V$ is the common vertex of two edges $E$ and $E'$, of $T$, namely that the two compositions $M(T) \to M(E) \to M(V)$ and $M(T) \to M(E') \to M(V)$ are equal. But, the composed restriction $M(T) \to M(E) \to M(V)$ is:
\begin{equation}
 \phi \mapsto (-\partial_\tau (\phi \cdot \nu), (\phi \cdot \nu) \tau - (\phi \cdot \tau)\nu ) = (1/2\curl \phi (V), - J \phi (V)),
\end{equation}
which is independent of $E$. 
\end{itemize}

We need bilinear pairings.
\begin{itemize}
\item On $\Gamma(T, \bbV) \times M(T)$ we define $\langle u, \phi \rangle_T = \int_T u \cdot \phi$.
\item On $\Gamma(E, \bbR \times \bbR) \times M(E)$ we define $\langle (u,v), (\phi, \psi) \rangle_E = \int_E u \phi + v \psi$.
\item On $\Gamma(V, \bbR \times \bbV) \times M(V)$ we define $\langle (u,v), (\phi, \psi) \rangle_V = u \phi + v \cdot \psi$.
\end{itemize} 

We also need two Stokes-like identities. They are:

\begin{itemize}
\item If $u \in \Gamma(T, \bbS)$ and $\phi \in M(T)$, we write:
\begin{align}
\int_T \div u \cdot \phi & = \int_{\partial T}  u \nu \cdot \phi,\\
& = \int_{\partial T} (u \nu \cdot \tau)(\phi \cdot \tau)  +  (u \nu \cdot \nu)(\phi \cdot \nu).
\end{align}
\item If $(u,v) \in \Gamma(E, \bbR \times \bbR)$ and $(\phi,\psi)\in M(E)$, we have:
\begin{align}
\int_E \diff^0_E (u,v) \cdot (\phi,\psi) &  = \int_E (- \partial_\tau v)\phi  + (\partial_\tau^2 u)\psi,\\
& = \lsb \partial_\tau u \psi - u \partial_\tau \psi - v \phi \rsb,\\
& = \lsb u(-\partial_\tau \psi)   + (\partial_\tau u \tau + v \nu) \cdot (\psi \tau - \phi \nu) \rsb.
\end{align}
\end{itemize}


\begin{lemma}\label{lem:kerneldofstress}\mbox{}
\begin{itemize}
\item For any triangle $T$, the complex $\calC^\bs(\subcells(T), M^\star)$ is exact except at index $0$ where the kernel can be characterized as follows. For $u \in \calC^0(\subcells(T), M^\star)$ we have $\delta_\trans u = 0$ iff $u$ represents the degrees of freedom of an affine function -- i.e. there exists an affine function $v$ on $T$, such that $u$ and $v$ evaluate similarly against $M(V)$ for each vertex $V$ of $T$.
\item  For any edge $E$, the complex $\calC^\bs(\subcells(E), M^\star)$ is exact except at index $0$ where the kernel can be characterized as follows. For $u \in \calC^0(\subcells(T), M^\star)$ we have $\delta_\trans u = 0$ iff $u$ represents the degrees of freedom of the restriction of an affine function -- i.e. there exists an affine $v$ on $\bbR^2$, such that $u$ and $v$ evaluate similarly against $M(V)$ for each vertex $V$ of $E$.
\end{itemize}
\end{lemma}
\begin{proof}
From Corollary \ref{cor:cohlt}. At index $0$ the kernel of the discrete covariant exterior derivative has dimension 3. On the other hand the affine functions naturally inject into this kernel. Surjectivity then follows from dimension equality.
\end{proof}

\subsection{Discrete spaces}
We now define finite element spaces for the stress complex, following \cite{JohMer78}. Our main novely is that by highlighting degrees of freedom involving the $M(T)$ spaces, we are led to a natural reduction of this space (where stresses have dimension 9 on an element, reduced from 15), see Remark \ref{rem:minjm}. Notice that in \cite{JohMer78}, the displacement is chosen in the finite element space $\rmP^1(T, \bbV)$ rather than $\rmP^0(\calR(T), \bbV)$. Thus their finite element pair is not part of a complex. A point of view emphasizing discrete complexes is developed in \cite{ArnDouGup84}.

For a triangle $T$ we denote by $\calR(T)$ the Clough-Tocher split of $T$.

\begin{definition}[FE for the stress complex]\mbox{}
\begin{itemize}
\item We define:
\begin{equation}
A^0(T)  = \rmC^1 \rmP^3 (\calR(T), \bbV),
\end{equation}
$A^0(T)$ is thus  the Clough-Tocher space. 
The degrees of freedom are:
values at vertices ($3$), values of the gradient at vertices ($3 \times 2$), integral of normal derivative on edges ($3$). 

\item We define, following \cite{JohMer78}:
\begin{equation}
A^1(T) =  \rmH^0_{\div} \rmP^1(\calR(T), \bbS),
\end{equation}
by which we mean symmetric matrix fields that are piecewise polynomial of degree at most $1$, with continuous normal components on interior edges. 
The degrees of freedom are: on edges E, $\int_E u\nu \cdot v$ for $v \in \rmP^1(E, \bbV)$ ($3 \times 4$), and on $T$, integration against $\rmP^0(T, \bbS)$ ($3$).
\item We define:
\begin{equation}
A^2(T)= \rmP^0(\calR(T), \bbV),
\end{equation}
that is the space of piecewise constant vector fields. The degrees of freedom are integration against $\rmP^1(T, \bbV)$ ($6$).
\end{itemize}
\end{definition}
These finite element spaces are represented in Figure \ref{fig:stress}.
\begin{figure}
\includegraphics[width = \textwidth]{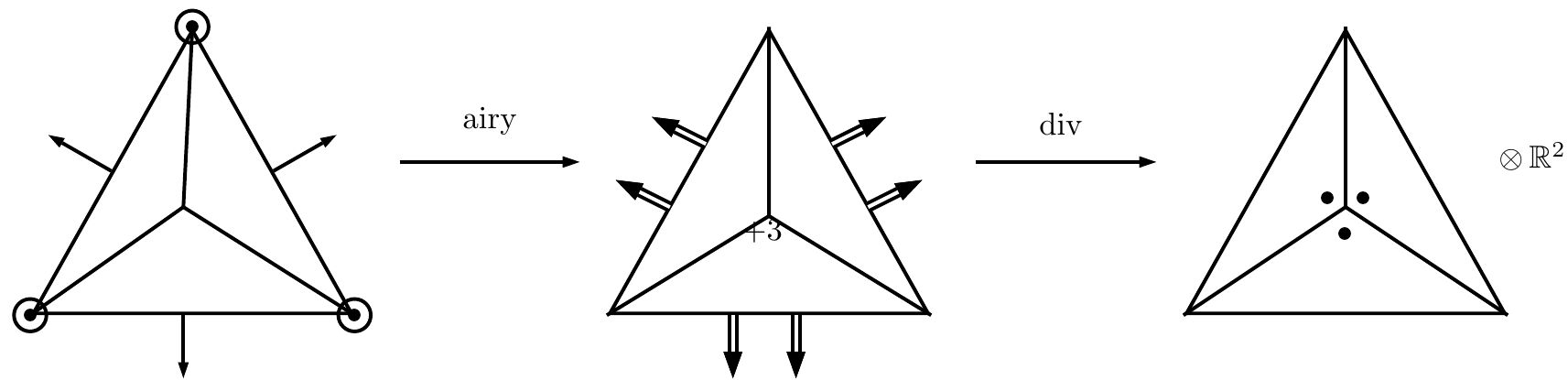}
\caption{Johnson Mercier stress complex. \label{fig:stress}}
\end{figure}

\begin{remark}\label{rem:mdofjm}
Notice that:
\begin{itemize}
\item for $A^0(T)$, pairings of restrictions with $M(V)$ at vertices $V$ can be recovered from the DoFs.
\item for $A^1(T)$, pairings of restrictions with $M(E)$ at edges $E$ can be recovered from the DoFs.
\item for $A^2(T)$, pairings with $M(T)$ can be recovered from the DoFs.
\end{itemize}
\end{remark}

\begin{theorem}[FE for the stress complex] \label{theo:stresscomplex} \mbox{}
\begin{itemize}
\item $A^0(T)$ has dimension $12$ and the provided DoFs are unisolvent.

\item $A^1(T)$ has dimension $15$ and the provided DoFs are unisolvent.

\item $A^2(T)$ has dimension $6$ and the provided DoFs are unisolvent.

\item The complex $A^\bs(T)$ is a resolution of the affine functions.
\end{itemize}
\end{theorem}
\begin{proof}
\tpoint For $A^0(T)$ this is well known.

\tpoint For $A^2(T)$ one can consider the scalar analogue, and go via the dual result, that any $u \in \rmP^1(T)$ is uniquely determined by the integrals on the three small triangles in $\calR(T)$, since their isobarycentres are not colinear.

\tpoint For $A^1(T)$ this is proved in \cite{JohMer78}, also via the Clough-Tocher element. We provide a slight modification of that proof, exploiting the $M$ DoFs (see Remark \ref{rem:mdofjm}).

The space $\rmP^1(\calR(T), \bbS)$ has dimension $3 \times 3 \times 3 = 27$. Imposing continuity of the normal component on interior edges can be expressed with $3 \times 2 \times 2 = 12$ constraints, so $\dim A^1(T)\geq 27-12 = 15$.

Now let $u \in A^1(T)$ and suppose that its DoFs are $0$. It then follows that $\div u \in A^2(T)$ has DoFs $0$, by integration by parts. So $\div u = 0$ and hence $u = \airy v$ for some $v \in \rmH^2(T)$. The second order derivatives of $v$ are in $\rmP^1(\calR(T))$, hence $v \in A^0(T)$. Now the $M(E)$-DoFs of $\airy v$ are $0$, so there exists $w \in \rmP^1(T)$ such that $v$ and $w$ have the same $M(V)$-DoFs (see Lemma \ref{lem:kerneldofstress}). We have $u = \airy(v - w)$. On a given edge $\partial_\tau \partial_\nu (v-w) = 0$ from the DoFs of $u$, so $\partial_\nu (v-w)$ is affine, in fact constant. Therefore $v-w = 0$ so $u= 0$.

\tpoint Exactness of $A^\bs(T)$ at index 1, was just proved. At index $0$ the kernel is the space of affine functions. Exactness at index $2$ then follows by dimension count.
\end{proof}

\begin{remark}
For each edge $E$, let $\chi_E$ be a nonzero affine map $E \to \bbR$ such that $ \int \chi_E = 0$.

In $A^1(T)$ we may think of the provided DoFs attached to an edge $E$ as
\begin{itemize}
\item Integrals of $u \nu \cdot \nu$ against $\rmP^1(E)$ and of $u \nu \cdot \tau$ against $\rmP^0(E)$, which together constitute pairings with $M(E)$.
\item Integral of $u \nu \cdot \tau$ against $\chi_E$.
\end{itemize}

From this point of view it seems natural to replace the edge DoF in $A^0(T)$ (namely $u \mapsto \int_E \partial_\nu u$) by $u \mapsto \int (\partial_\tau \partial_\nu u) \chi_E$. In particular $\int (\partial_\tau \partial_\nu u) \chi_E = 0$ iff $\partial_\nu u$ is affine, which appeared as a step in the proof of Theorem \ref{theo:stresscomplex}.
\end{remark}

\begin{definition}[FES for the stress complex]
We get a finite element system by appending the following spaces:
\begin{align}
A^0(E)& = \rmP^3(E) \times \rmP^2(E),\\
A^1(E)& = \rmP^1(E) \times \rmP^1(E),\\
A^0(V)& = \bbR \times \bbV.
\end{align}
A system of degrees of freedom is defined by:
\begin{align}
F^0(T) & = 0,\\
F^0(E) & = \{A^0(E) \ni (u,v) \mapsto \int_E (\partial_\tau v) \chi_E \},\\
F^0(V) & = \{\langle \cdot, \phi \rangle_V \ : \ \phi \in M(V) \},\\
F^1(T) & = \{A^1(T) \ni u \mapsto \int_T u \cdot v \ : \ v \in \rmP^0(T, \bbS)\},\\
F^1(E) & = \{\langle \cdot, \phi \rangle_E \ : \ \phi \in M(E) \} \oplus\\
& \qquad \bbR \{A^1(E) \ni (u,v) \mapsto \int_E u \chi_E \},\\
F^2(T) & = \{A^2(T) \ni u \mapsto \int_T u \cdot v \ : \ v \in \rmP^1(T, \bbV) \}.
\end{align}
\end{definition}
\begin{proposition}
The above finite element system is compatible and the system of degrees of freedom unisolvent.
\end{proposition}
\begin{proof}
The essential points were already proved in Theorem \ref{theo:stresscomplex}. The main addition is the unisolvence of the degrees of freedom on $A^1(E)$, which is straigthforward.
\end{proof}

\begin{remark}[Minimal spaces]\label{rem:minjm}
\mbox{}
The preceding FE complex may be reduced as follows:
\begin{itemize}
\item $\tilde A^0(T)$ is the reduced Clough-Tocher element where $\partial_\nu u$ is affine on edges. The DoFs are now only pairings with $M(V)$ at vertices. The dimension is $9$.
\item $\tilde A^2(T)$ consists of divergence free elements of $\rmP^0(\calR(T))$, i.e. the elements that are continuous in the normal direction on interior edges. It can also be characterized as $\curl \rmC^0\rmP^1(\calR(T), \bbR)$. The dimension is $3$ and the degrees of freedom are now pairings with $M(T)$ only.
\item $\tilde A^1(T)$  is the subspace of $\rmH^0_{\div} \rmP^1(\calR(T), \bbS)$, consisting of elements $u$ such that $\div u \in \tilde A^2(T)$ and, on any edge, $u \nu \cdot \tau \in \rmP^0(E)$. The dimension is $9$ and the degrees of freedom are pairings with $M(E)$ on edges, only.
\end{itemize}
A compatible FES is obtained by appending
\begin{align}
A^0(E)& = \rmP^3(E) \times \rmP^1(E),\\
A^1(E)& = \rmP^0(E) \times \rmP^1(E),\\
A^0(V)& = \bbR \times \bbV.
\end{align}
As already asserted, natural degrees of freedom are:
\begin{align}
\tilde F^0(T) & = 0,\\
\tilde F^0(E) & = 0,\\
\tilde F^0(V) & = \{\langle \cdot, \phi \rangle_V \ : \ \phi \in M(V) \},\\
\tilde F^1(T) & = 0,\\
\tilde F^1(E) & = \{\langle \cdot, \phi \rangle_E \ : \ \phi \in M(E) \} \\
\tilde F^2(T) & = \{\langle \cdot, \phi \rangle_E \ : \ \phi \in M(T) \}.
\end{align} 
\end{remark}

\begin{remark} 
Discrete diagram chase for the Johnson Mercier stress complex. See Figure \ref{fig:diagstress}.

One could also exhibit an alternative diagram chase, based on Remark \ref{rem:chasestressbis}.
\end{remark}

\begin{figure}
\includegraphics[width= 16cm]{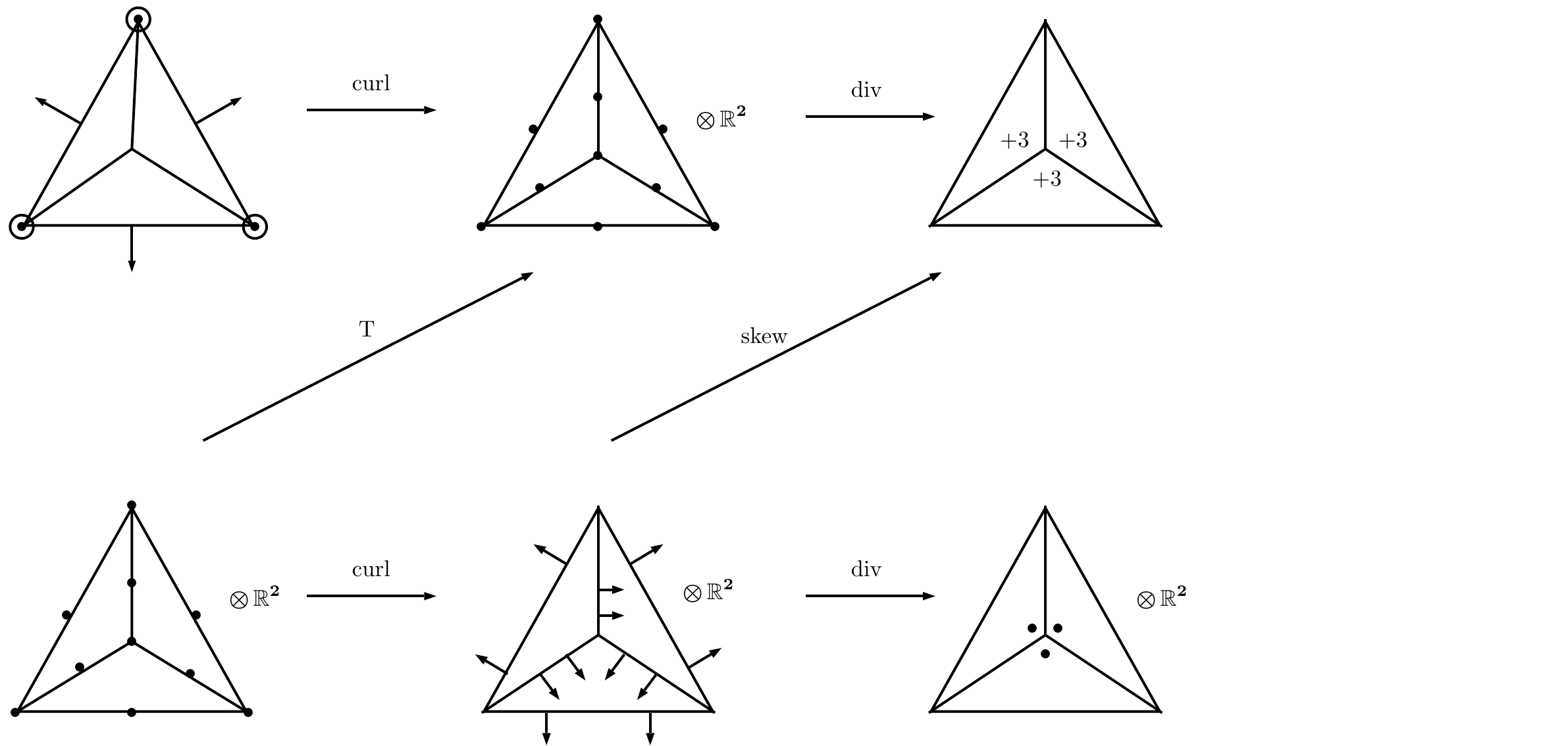}
\caption{Diagram chase for the Johnson Mercier element. \label{fig:diagstress}}
\end{figure}

\begin{remark}
High order finite element stress complexes are provided in \cite{ArnDouGup84}. They seem to fit in our framework too. Their lowest order complex starts with $\rmC^1\rmP^4(\calR(T), \bbR)$ and ends in $\rmP^1(T, \bbV)$. The middle space of stresses augments $\rmP^2(T, \bbS)$ with a 3-dimensional space, so has dimension 21.
\end{remark}

\section{FES for the strain complex \label{sec:fesstrain}}

\subsection{Induced operators and discrete vector bundle}

We here consider now the complex (\ref{eq:elastrain}). We identify induced spaces and operators on edges and vertices. They are summarized in the commuting diagram depicted in Figure \ref{fig:straindiagram}.

Again we consider a vertex $V$, in an edge $E$, in a triangle $T$. The oriented unit tangent on $E$ is denoted $\tau$ and the normal is denoted $\nu$, so that $(\tau , \nu)$ is an oriented orthonormal basis of $\bbV$.

\begin{figure}[htbp]
\begin{tikzpicture}

\matrix (L) [matrix of math nodes, row sep = 2.5cm, column sep = 2cm] {
\Gamma(T, \bbV) & \Gamma(T, \bbS) & \Gamma(T, \bbR) \\
\Gamma(E, \bbR^2 \times \bbR^2) & \Gamma(E, \bbR^3 \times \bbR) \\
\Gamma(V, \bbV \times \bbM) & \Gamma(V, \bbS) \\
};

\path[->] 
(L-1-1) edge node[above] {$\defo$} (L-1-2)
(L-1-2) edge node[above] {$\sven$} (L-1-3)
(L-2-1) edge (L-2-2)
(L-1-1) edge (L-2-1)
(L-2-1) edge (L-3-1)
(L-1-2) edge (L-2-2)
(L-2-2) edge (L-3-2)
(L-3-1) edge (L-3-2)
(L-1-1) edge [out= -130, in = 130]  node [left]{\footnotesize $\begin{array}{c}
u \\ \downarrow \\ (u, \grad u) \end{array}$} (L-3-1)
;

\draw (L-1-2) + (0, -0.7) node [below right] {\footnotesize $\begin{array}{l}
\quad u \\ \quad \downarrow \\ (u\tau \cdot \tau, u \tau \cdot \nu, u \nu \cdot \nu,  \partial_\nu u\tau\cdot \tau) \end{array}$};

\draw (L-1-1) + (0, -0.7) node [below right] {\footnotesize $\begin{array}{l} 
\quad u\\ \quad \downarrow \\ (u\cdot \tau, u \cdot \nu, \partial_\nu u\cdot \tau, \partial_\nu u \cdot \nu )
\end{array}$};

\draw (L-2-1) + (0, -0.2) node [below right] {\footnotesize $\begin{array}{l}
(u, v , u', v')  \rightarrow  (\partial_\tau u, \frac{1}{2}(u' + \partial_\tau v), v' , \partial_\tau u') \\
\quad \downarrow\\
 (u \tau + v \nu, \partial_\tau u \tau \tau^\transp + \partial_\tau v \nu \tau^\transp + u' \tau \nu^\transp + v' \nu \nu^\transp ) 
\end{array}$};

\draw (L-2-2) + (0, -1.5) node [below right] {\footnotesize $\begin{array}{l}
(u,v,w, k)\\
\quad \downarrow\\
\quad u \tau \tau^\transp + v (\tau \nu^\transp + \nu \tau^\transp) + w \nu \nu^\transp
\end{array}$};

\draw (L-3-1) + (0, -0.2) node [below right] {\footnotesize $\begin{array}{l}
(u,v) \rightarrow \sym(v) 
\end{array}$};

\end{tikzpicture}
\caption{Straincomplex and induced operators: high regularity \label{fig:straindiagram}}
\end{figure}

If we consider instead the complex of lower regularity, as in (\ref{eq:elastrainlow}), the corresponding diagram is depicted in Figure \ref{fig:straindiagramlow}.

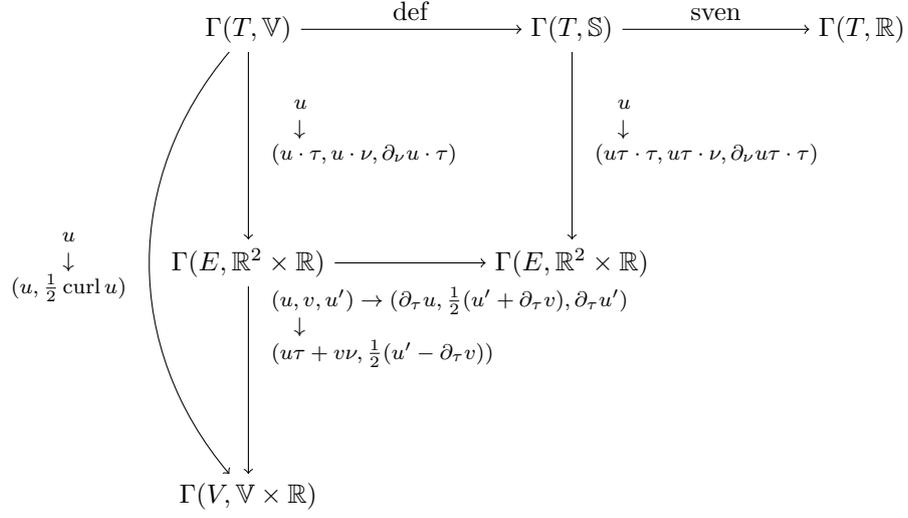
\begin{figure}[htbp]

\begin{tikzpicture}

\matrix (L) [matrix of math nodes, row sep = 2.5cm, column sep = 2cm] {
\Gamma(T, \bbV) & \Gamma(T, \bbS) & \Gamma(T, \bbR) \\
\Gamma(E, \bbR^2 \times \bbR) & \Gamma(E, \bbR^2 \times \bbR) \\
\Gamma(V, \bbV \times \bbR) \\
};

\path[->] 
(L-1-1) edge node[above] {$\defo$} (L-1-2)
(L-1-2) edge node[above] {$\sven$} (L-1-3)
(L-2-1) edge (L-2-2)
(L-1-1) edge (L-2-1)
(L-2-1) edge (L-3-1)
(L-1-2) edge (L-2-2)
(L-1-1) edge [out= -130, in = 130]  node [left]{\footnotesize $\begin{array}{c}
u \\ \downarrow \\ (u, \frac{1}{2}\curl u) \end{array}$} (L-3-1)
;

\draw (L-1-2) + (0, -0.7) node [below right] {\footnotesize $\begin{array}{l}
\quad u \\ \quad \downarrow \\ (u\tau \cdot \tau, u \tau \cdot \nu,  \partial_\nu u\tau\cdot \tau) \end{array}$};

\draw (L-1-1) + (0, -0.7) node [below right] {\footnotesize $\begin{array}{l} 
\quad u\\ \quad \downarrow \\ (u\cdot \tau, u \cdot \nu, \partial_\nu u\cdot \tau)
\end{array}$};

\draw (L-2-1) + (0, -0.2) node [below right] {\footnotesize $\begin{array}{l}
(u, v , u')  \rightarrow  (\partial_\tau u, \frac{1}{2}(u' + \partial_\tau v), \partial_\tau u') \\
\quad \downarrow\\
 (u \tau + v \nu, \frac{1}{2}(u' - \partial_\tau v) ) 
\end{array}$};

\end{tikzpicture}

\caption{Straincomplex and induced operators: low regularity \label{fig:straindiagramlow}}
\end{figure}

The discrete vector bundle is defined as follows:

\begin{itemize}
\item We let $M(T)$ be the kernel of the Airy operator, namely the space of affine functions.
\item We let $M(E)$ the space of pairs $(u,v)$ where $u$ is an affine function on $E$ and $v$ is constant.
\item We let $M(V)$ be the space $\bbV \times \bbR$.
\end{itemize}

We define restriction operators:
\begin{itemize}
\item $M(T) \to M(E) : \phi \mapsto (\phi|_E, \partial_\nu \phi|_E)$.

\item $M(E) \to M(V) : (\phi, \psi) \mapsto (\psi(V) \tau - \partial_\tau \phi(V) \nu, \phi(V))$.

\item The composition $M(T) \to M(E) \to M(V)$  is then $\phi \mapsto (\curl \phi (V), \phi (V))$, which is independent of $E$. 
\end{itemize}

The pairings are:
\begin{itemize}
\item on $\Gamma(T, \bbR) \times M(T)$ : $\langle u, \phi \rangle = \int u \phi$.
\item on $\Gamma(E, \bbR^2 \times \bbR) \times M(E)$: 
\begin{equation}
\langle (u,v, u') , (\phi, \psi) \rangle = \int_E u \psi + \partial_\tau v \phi - v \partial_\tau \phi - u' \phi. 
\end{equation}
\item on $\Gamma(V, \bbV \times \bbR) \times M(V)$ : $\langle (u, v) , (\phi, \psi) \rangle =  u \cdot \phi + v \psi$.
\end{itemize}

The Stokes-like identities are:

\begin{itemize}
\item For $u \in  \Gamma(T, \bbS)$ and $\phi \in M(T)$ we have:
\begin{align}
\int_T \sven u \ \phi & = \int_{\partial T} \curl u \cdot \tau \phi + \int_{\partial T} u \tau \cdot  \transp \curl \phi, \\
& = \int_{\partial T} (\partial_\tau u \tau \cdot \nu - \partial_\nu  u \tau \cdot \tau) \phi + u \tau \cdot (\partial_\nu \phi \tau - \partial_\tau \phi \nu),\\
& = \int_{\partial T} u \tau \cdot \tau \partial_\nu \phi + \partial_\tau u\tau \cdot \nu \phi - u \tau \cdot \nu \partial_\tau \phi - \partial_\nu u \tau \cdot \tau  \phi,\\
& = \sum_E \langle (u \tau \cdot \tau, u \tau \cdot \nu, \partial_\nu u\tau \cdot \tau), (\phi, \partial_\nu \phi) \rangle_E.
\end{align}

 \item For $(u,v,u') \in \Gamma (E, \bbR^2 \times \bbR)$ and $(\phi, \psi) \in M(E)$:
\begin{align}
& \phantom{=} \langle (\partial_\tau u, \frac{1}{2}(u' + \partial_\tau v), \partial_\tau u'), (\phi, \psi) \rangle ,\\
& = \int_E \partial_\tau u \psi + \partial_\tau \frac{1}{2}(u' + \partial_\tau v) \phi -  \frac{1}{2}(u' + \partial_\tau v) \partial_\tau \phi - \partial_\tau u' \phi, \\
& = \lsb u \psi \rsb  + \lsb \frac{1}{2}(u' + \partial_\tau v) \phi \rsb  - \int (u' + \partial_\tau v) \partial_\tau \phi - \lsb u' \phi \rsb  + \int u' \partial_\tau \phi,\\
& = \lsb u \psi \rsb  + \lsb \frac{1}{2}(u' + \partial_\tau v) \phi \rsb  - \lsb  v \partial_\tau \phi \rsb  - \lsb u' \phi \rsb ,\\
& = \lsb u \psi \rsb  + \lsb \frac{1}{2}(\partial_\tau v -u') \phi \rsb  - \lsb  v \partial_\tau \phi \rsb ,\\
& = \lsb  \langle (u \tau + v \nu, \frac{1}{2}(\partial_\tau v -u')) , (\psi \tau - \partial_\tau \phi \nu, \phi) \rangle_{V} \rsb .
\end{align}
Here the brackets denote differences between values at the two vertices of $E$.
\end{itemize}

\begin{lemma}\label{lem:kerneldof}\mbox{}
\begin{itemize}
\item For any triangle $T$, the complex $\calC^\bs(\subcells(T), M^\star)$ is exact except at index $0$ where the kernel can be characterized as follows. For $u \in \calC^0(\subcells(T), M^\star)$ we have $\delta_\trans u = 0$ iff $u$ represents the degrees of freedom of a rigid motion -- i.e. there exists a rigid motion $v$ on $T$, such that $u$ and $v$ evaluate similarly against $M(V)$ for each vertex $V$ of $T$.
\item  For any edge $E$, the complex $\calC^\bs(\subcells(E), M^\star)$ is exact except at index $0$ where the kernel can be characterized as follows. For $u \in \calC^0(\subcells(T), M^\star)$ we have $\delta_\trans u = 0$ iff $u$ represents the degrees of freedom of the restriction of of rigid motion -- i.e. there exists a rigid motion $v$ on $\bbR^2$, such that $u$ and $v$ evaluate similarly against $M(V)$ for each vertex $V$ of $E$.
\end{itemize}
\end{lemma}
\begin{proof}
From Corollary \ref{cor:cohlt}. At index $0$ the kernel of the discrete covariant exterior derivative has dimension 3. On the other hand the rigid motions naturally inject into this kernel. Surjectivity then follows from dimension equality.
\end{proof}

\subsection{Discrete spaces: higher regularity}

For a triangle $T$ we denote by $\calR(T)$ the Clough Tocher split of $T$.

\begin{definition}[FE for the strain complex: definition \`a la Ciarlet] \label{def:festrain} \mbox{}
\begin{itemize}
\item We define:
\begin{equation}
A^0(T)  = \rmC^1 \rmP^3 (\calR(T), \bbV),
\end{equation}
$A^0(T)$ is thus the vector valued variant of the Clough Tocher space. 
The degrees of freedom are:
values at vertices ($3 \times 2$), values of the gradient at vertices ($3 \times 4$), integral of normal derivative on edges ($3 \times 2$). 

\item We define:
\begin{equation}
A^1(T) =  \rmC^0_{\sven} \rmP^2(\calR(T), \bbS),
\end{equation}
by which we mean symmetric matrix fields that are piecewise polynomial of degree at most $2$, are continuous, and have an integrable $\sven$. 
The degrees of freedom are: values at vertices ($3 \times 3$), pairings of edge restrictions with $M(E)$ for each edge $E$ ($3 \times 3$), integral against normal vector on edges ($3 \times 2$).

\item We define:
\begin{equation}
A^2(T)= \rmP^0(\calR(T), \bbR),
\end{equation}
that is the space of piecewise constants. The degrees of freedom are pairings with $M(T)$, namely integration against affine functions.
\end{itemize}
\end{definition}

\begin{remark}
On $A^0(T)$, pairings of restrictions to vertices, with $M(V)$ at vertices $V$, providing the vertex values and vertex values of the $\curl$, constitute a subspace of the provided vertex DoFs.

On $A^1(T)$, the edge DoFs involve normal derivatives through the restriction operators and tangential derivatives through the pairing with $M(E)$.
\end{remark}

The finite element spaces of Definition \ref{def:festrain} are represented in Figure \ref{fig:strainhigh}.
\begin{figure}
\includegraphics[width= \textwidth]{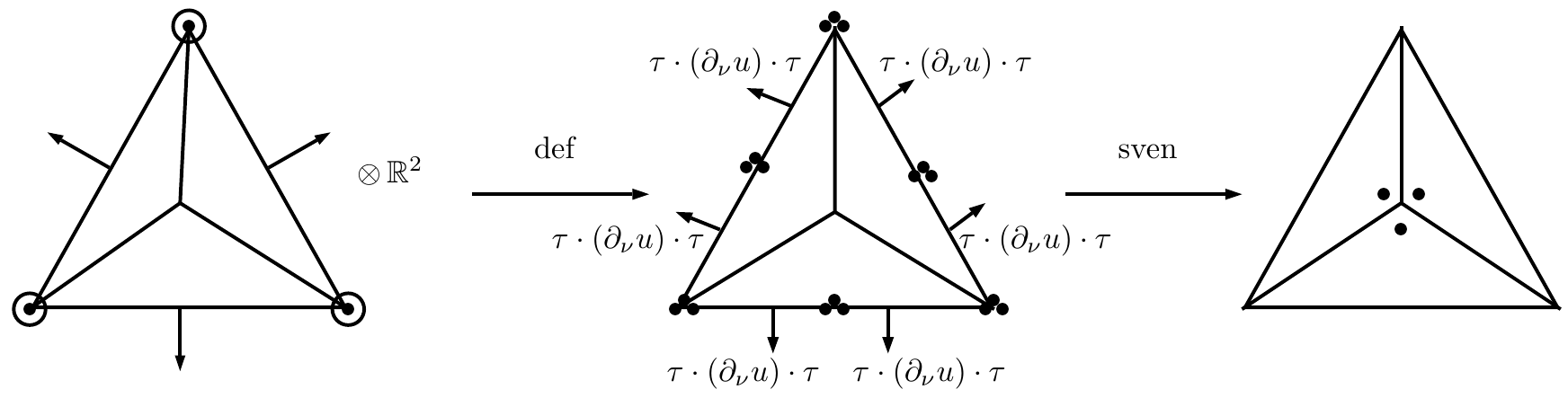}
\caption{Strain complex with high continuity. \label{fig:strainhigh}}
\end{figure}

\begin{proposition}
The provided DoFs give interpolators commuting with the differential operators.
\end{proposition}

The preceding finite element spaces can regarded as a finite element system, as follows.
\begin{definition}[FES for the strain complex] \label{def:fesstrain} \mbox{}
\begin{itemize}
\item Differential operators and restrictions are defined according to Figure \ref{fig:straindiagram}.

\item The spaces $A^0(T)$, $A^1(T)$ and $A^2(T)$ are defined as in Definition \ref{def:festrain}.

\item Spaces on edges $E$ and vertices $V$ are defined by:
\begin{align}
A^0(E) & = \rmP^3(E) \times \rmP^3(E) \times \rmP^2(E) \times \rmP^2(E), \\ 
A^1(E) & = \rmP^2(E) \times \rmP^2(E) \times \rmP^2(E) \times \rmP^1(E),  \\
A^0(V) & = \bbV \times \bbM, \\
A^1(V) & = \bbS.
\end{align}

\item A system of degrees of freedom $F$ on $A$ is defined by:
\begin{align}
F^0(T) & = 0, \\
F^0(E) & = \{A^0(E) \ni (u,v,u',v') \mapsto \ts \int_E u' \} \oplus\\
       & \qquad \{A^0(E) \ni (u,v,u',v') \mapsto \ts \int_E v' \}, \\
F^0(V) & = A^0(V)^\star \approx \bbV^\star \oplus \bbM^\star, \\
F^1(T) & = 0, \\
F^1(E) & = \{ \langle \cdot , \phi \rangle_E \ : \ \phi \in M(E) \} \oplus\\ 
       & \qquad \bbR \{A^1(E) \ni (u,v,w, u') \mapsto \ts \int_E v \} \oplus\\
       & \qquad \bbR \{A^1(E) \ni (u,v,w, u') \mapsto \ts \int_E w \} , \\
F^1(V) & = A^1(V) = \bbS^\star, \\
F^2(T) & = \{ \langle \cdot , \phi \rangle_T \ : \ \phi \in M(T) \}.
\end{align}
(other spaces are set to $0$).
\end{itemize}
\end{definition}

\begin{theorem}[FE for the strain complex] \label{theo:straincomplex} \mbox{}
\begin{itemize}
\item $A^0(T)$ has dimension $24$ and the provided DoFs are unisolvent.

\item $A^1(T)$ has dimension $24$ and the provided DoFs are unisolvent.

\item $A^2(T)$ has dimension $3$ and the provided DoFs are unisolvent.

\item The sequence $A^\bs(T)$ is a resolution of the rigid motions.
\end{itemize}

\end{theorem}

\begin{proof}
\tpoint For $A^0(T)$ the dimension count, and the unisolvence of the degrees of freedom are standard. We also remark that the degrees of freedom corresponding to an edge $E$ are unisolvent on $A^0(E)$. At vertices the corresponding result is trivial.

\tpoint For $A^2(T)$ the dimension count is trivial and the unisolvence of the DoFs is straightforward.

\tpoint The space $\rmC^0 \rmP^2(\calR(T), \bbS)$ has dimension $10 \times 3 = 30$. For $u \in \rmC^0 \rmP^2(\calR(T), \bbS)$, in order to impose that $\sven u$ is integrable we impose that $\partial_\nu u \tau \cdot \tau$ is continuous on the three interior edges. Since these fields are linear, this can be expressed as $6$ constraints. This shows that $\dim A^1(T) \geq 24$. 

\tpoint We now prove unisolvence of the degrees of freedom for $A^1(T)$. Let $u \in A^1(T)$ and suppose that the DoFs are all $0$. Then $\sven u \in A^2(T)$ has $0$ degrees of freedom, so it is $0$. So we can choose $v \in \rmH^2(T, \bbV)$ so that $\defo v = u$. The second order derivatives of $v$ can be recovered from the first order derivatives of $u$, and therefore turn out to be linear. Therefore $v \in A^0(T)$.

Since the $M(E)$-dofs of $\defo v$ are $0$, there exists a rigid motion $w$, which has the same $M(V)$-dofs as $w$ (ie value and value of $\curl$), by Lemma \ref{lem:kerneldof}. We notice that $\defo(v-w) = u$ and proceed to show that $v-w = 0$, by showing that its degrees of freedom, as defined in $A^0(T)$, are $0$.

We have that $v-w$ is $0$ at vertices and that $\grad (v-w)$ is $0$ at vertices (the symmetric part is $\defo v = u$ and antisymmetric parts is essentially $\curl (v - w)$).

It remains to prove that the integral of the normal derivative of $v -w$ on edges is $0$: 

-- We have that $\int_E \partial_\nu (v-w) \cdot \nu = 0$ since this is one of the DOFs of $u$.

-- We have that $\int_E \partial_\tau (v-w) \cdot \nu =0$  by integration of a derivative. Since $\int_E u \tau \cdot \nu = 0$ it follows that $\int_E \partial_\nu (v-w) \cdot \tau =0$.

\tpoint This shows that $\dim A^1 (T) = 24$. We have also showed that the sequence $A^\bs(T)$ is exact at index $1$. It follows that the range of $\sven$ on $A^1(T)$ has dimension $\dim A^1(T) - \dim A^0(T) + \dim RM = 24 -24 + 3 = 3$. Therefore $\sven :  A^1(T) \to A^2(T)$ is surjective.
\end{proof}

\begin{remark} The preceding proof of unisolvence for $A^1(T)$ was written from the point of view of Definition \ref{def:festrain}. From the point of view of FES, as in the extended Definition \ref{def:fesstrain}, one would go via Proposition \ref{prop:suffunisolve}, with similar arguments. That way yields the additional important information that degrees of freedom attached to edges are unisolvent on the space of restrictions to the edge, which guarantees the appropriate global continuity of the finite element fields defined piecewise by their DoFs.
\end{remark}

\begin{remark}
Discrete diagram chase for the finite element strain complex with high regularity. See Figure \ref{fig:diagstrainhigh}.
\end{remark}

\begin{figure}
\includegraphics[width= \textwidth]{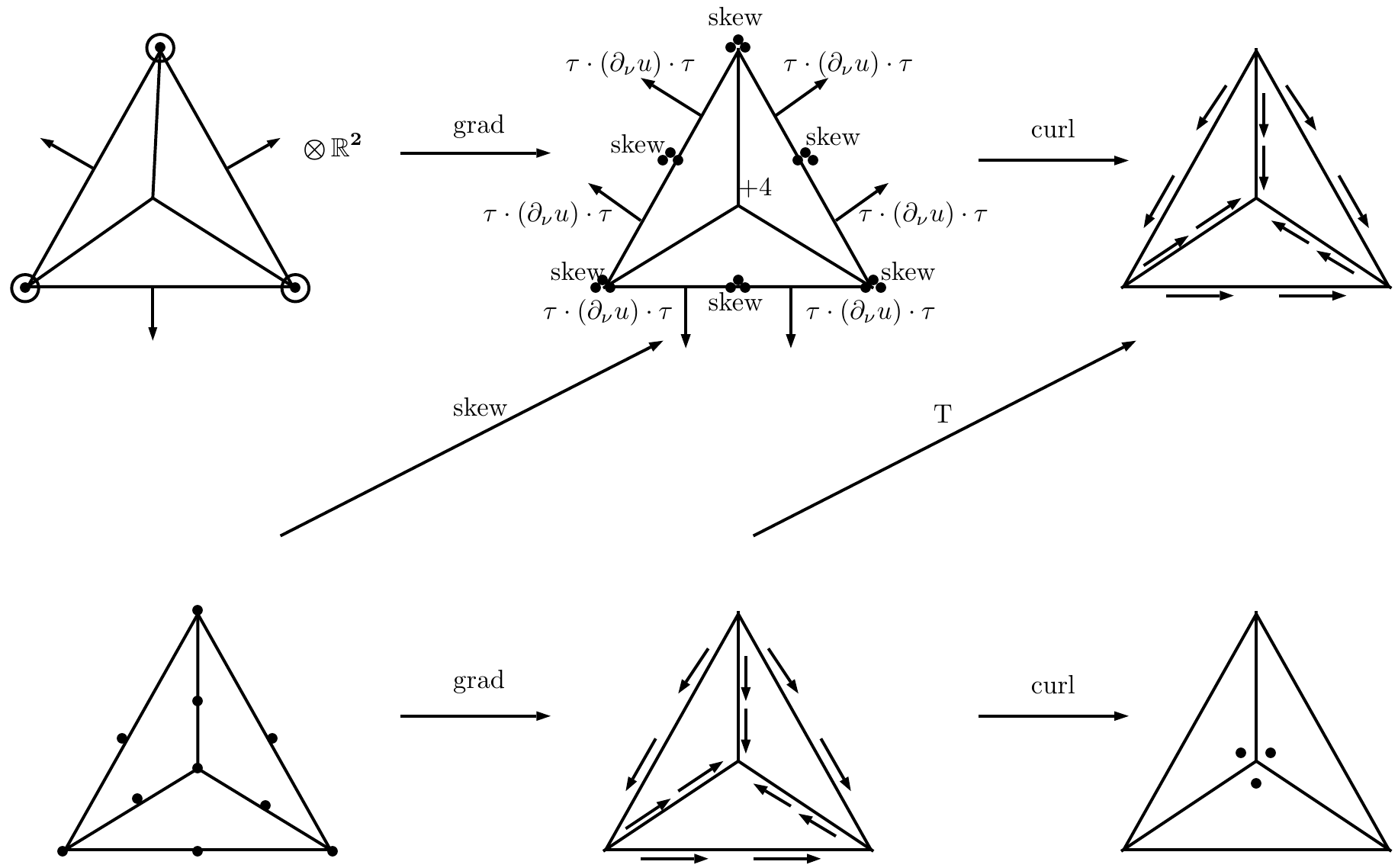}
\caption{Diagram chase for the strain element with high regularity. \label{fig:diagstrainhigh}}
\end{figure}

In the next section we will use the following consequence of Theorem \ref{theo:straincomplex}:
\begin{proposition}\label{prop:wt}
For any $v \in A^2(T)$ there is a unique $u \in A^1(T)$ such that $\sven u = v$, the restriction of $u$ to $\partial T$ is zero and $\partial_\nu u\tau \cdot \tau$ is constant on each edge.
\end{proposition} 
\begin{proof}
Indeed if we choose a constant $c_E$ for each edge $E$, the data $(0,0,0,c_E) \in A^1(E)$ is compatible at vertices, so can we extended to an element $u$ of $A^1(T)$, which is unique since there are no interior degrees of freedom. The Stokes identity then takes the form, for any affine $\phi$ on $T$:
\begin{equation}
\int \sven u \phi = - \sum_E c_E \int_E \phi.
\end{equation}
For any desired $v= \sven u \in A^2(T)$, this uniquely determines the coefficients $c_E$ of $u$.
\end{proof}

\begin{remark}[Minimal finite element strain complex with high regularity] One can get a minimal complex as follows.

We start with the modified reduced Clough Tocher space for $A^0(T)$. Recall that one usually requires the normal derivative on edges to be affine. Instead we take the subspace of vectorfields $u$ such that $\defo u$ applied to the normal vector on edges is affine. The degrees of freedom are just vertex values and vertex values of the gradient.

For $A^1(T)$ one takes the sum of the space $\defo A^0(T)$ and the space defined in the preceding proposition, so that normal components on edges are affine. The degrees of freedom consisting of vertex values and pairing with $M(E)$ for each edge $E$ are then unisolvent. Then $A^1(T) = \rmP^2(E) \times P^1(E) \times P^1(E) \times \rmP^1(E)$.

The space $A^2(T)$ is unchanged.

The canonical DoFs give interpolators that commute with the differential operators.
\end{remark}

\subsection{Discrete spaces: lower regularity}

In the following, for each edge $E$ we let $\chi_E$ be a nonzero affine map $E \to \bbR$ such that $ \int \chi_E = 0$. Notice that if $u : E \to \bbR$ is affine and $\int_E u \chi_E = 0$ then $u$ is constant.

\begin{definition}[FE for the strain complex: low regularity]\label{def:festrainlow}
\mbox{}
\begin{itemize}
\item $A^0(T) = \rmC^0_{\curl \transp} \rmP^2 (\calR(T), \bbV)$. The degrees of freedom are:\\ 
-- at vertices, pairings of restrictions to vertices with $M(V)$,
in other words, vertex values ($3 \times 2$) and vertex values of the curl ($3 \times 1$),\\
-- at edges, $u \mapsto \int_E \defo u \tau \cdot \tau \chi_E$ and $ u \mapsto \int_E \defo u \tau \cdot \nu \chi_E$ ($3 \times 2$).
\item $A^1(T) = \defo A^0(T) \oplus W(T)$, where $W(T)$ is the space defined in Proposition \ref{prop:wt}. The degrees of freedom are, for each edge $E$, pairings of restrictions with $M(E)$ ($3 \times 3$), and $u \mapsto \int_E u \tau \cdot \tau \chi_E$ and $u \mapsto \int_E u \tau \cdot \nu \chi_E$ ($3 \times 2$).
\item $A^2(T)  =  \rmP^0(\calR(T), \bbR)$, the space of piecewise constants. The degrees of freedom are pairings with $M(T)$, namely integration against affine functions (3).
\end{itemize}

\end{definition}

The finite element spaces of Definition \ref{def:festrainlow} are represented in Figure \ref{fig:strainlow}.
\begin{figure}
\includegraphics[width= \textwidth]{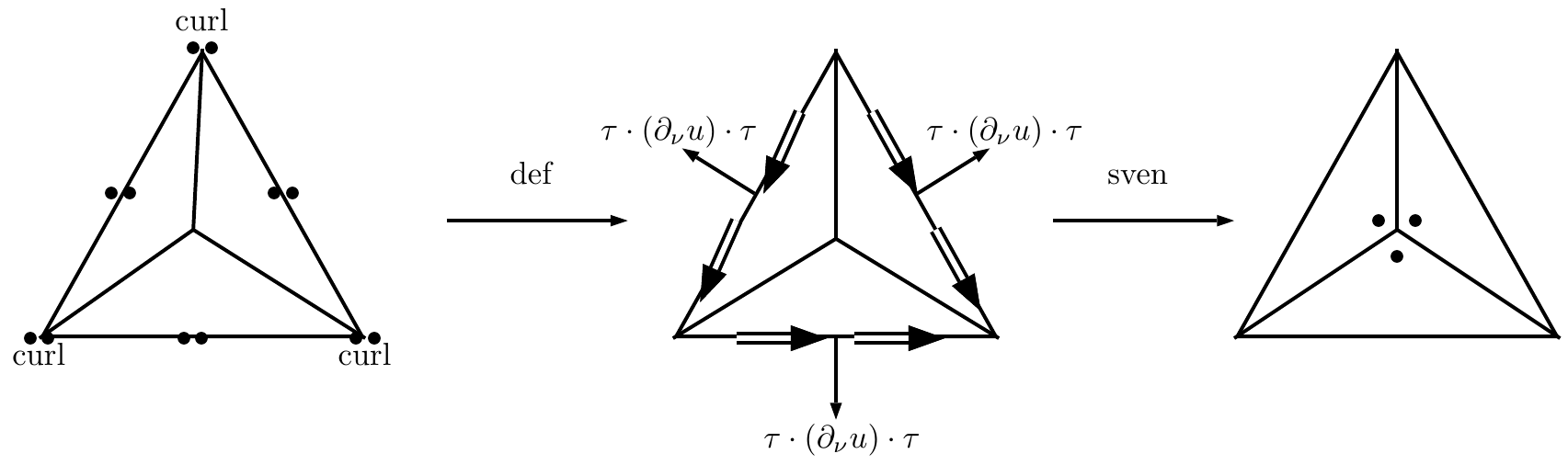}
\caption{Strain complex with low continuity. \label{fig:strainlow}}
\end{figure}

\begin{theorem} \mbox{}
\begin{itemize}
\item $A^0(T)$ has dimension 15, and the provided degrees of freedom are unisolvent.
\item $A^1(T)$ has dimension 15, and the provided degrees of freedom are unisolvent.
\item $A^2(T)$ has dimension 3, and the provided degrees of freedom are unisolvent.
\item The sequence $A^\bs(T)$ resolves the rigid motions.
\end{itemize}
\end{theorem}

\begin{proof}
\tpoint We introduced this space in \cite{ChrHu18} (Proposition 3), but there we had different edge degrees of freedom (in particular we had $u \mapsto \int_E u \cdot \tau$). In any case this gives the dimension. We now check unisolvence. Let $u \in A^0(T)$ have $0$ DoFs. In particular $u$ is $0$ at vertices as well as $\curl u$. Given preceding results it remains to be proved that $u$ is $0$ on each edge.

-- We have that $\defo u \tau \cdot \tau = \partial_\tau (u \cdot \tau)$ is affine and orthogonal to $\chi_E$ hence constant. So $u\cdot \tau$ is linear on each edge, hence $0$.
   
-- 
We have that $\defo u \tau \cdot \nu = (1/2) (\partial_\tau (u \cdot \nu) + \partial_\tau (u \cdot \nu)) $ affine and orthogonal to $\chi_E$, hence constant. Also $\partial_\tau (u \cdot \nu) - \partial_\tau (u \cdot \nu) = \curl \transp u = 0$ on $E$, so $\partial_\tau (u \cdot \nu)$ is constant so $u \cdot \nu = 0$ on each edge.  

\tpoint The dimension of $A^1(T)$ is (15 - 3) + 3 = 15 by construction.

\tpoint Unisolvence on $A^1(T)$. Choose $u \in A^1(T)$ with $0$ degrees of freedom. We then get $\sven u = 0$ from the Stokes identity, since the $M(E)$ DoFs are zero. So we may choose $v \in A^0(T)$ such that $u = \defo v$. We may find a rigid motion such that $v$ and $w$ have the same $M(V)$ degrees of freedom. Then $\defo (v-w) = u$ moreover the vertex degrees of freedom of $v-w$ are $0$. The remaining edge degrees of freedom of $v-w$, are edge degrees of freedom of $u$, hence $0$, so $v-w  = 0$. Hence $u = 0$.

\end{proof}

\begin{proposition}
We get a compatible FES by appending the spaces:
\begin{align}
A^0(E) & = \rmP^2(E) \times \rmP^1(E),\\
A^1(E) & = \rmP^1(E) \times \rmP^1(E) \times \rmP^0(E),\\
A^0(V) & = \bbV \times \bbR.
\end{align}

The degrees of freedom are now described as:
\begin{align}
F^0(V) & = \{A^0(V) \ni (u,v) \mapsto u \cdot \phi + v \psi \ : \ (\phi, \psi) \in M(V) \},\\
F^0(E) & = \bbR \{A^1(E) \ni (u,v,u') \mapsto \ts \int_E \partial_\tau u \xi_E \} \oplus \\
       & \qquad \bbR \{A^1(E) \ni (u,v,u') \mapsto \ts \int_E\frac{1}{2}(u' + \partial_\tau v) \xi_E \},\\
F^1(E) & = \bbR \{A^1(E) \ni u \mapsto \langle u, \phi \rangle_E \ : \ \phi \in M^1(E) \} \oplus \\
       & \qquad \bbR \{A^1(E) \ni (u,v,u') \mapsto \ts \int u \chi_E \} \oplus\\
       & \qquad \bbR \{A^1(E) \ni (u,v,u') \mapsto \ts \int v \chi_E \},\\
F^2(T) & = \{A^2(T) \ni u \mapsto \ts \int_T u \phi \ : \ \phi \in M(T) \}.
\end{align}

\end{proposition}

\begin{proof} 
Unisolvence of the edge DoFs on the edge spaces can be checked by similar arguments.
\end{proof}

\begin{remark}[Alternative definition of $W(T)$] One can replace the space $W(T)$ in Definition \ref{def:festrainlow} by a construction with the Poincar\'e - Koszul operators. Indeed let $x^\rota$ be the the identity vector field, rotated by $\pi/2$, with respect to an origin located at the central vertex of $\calR(T)$.  Consider the matrix field $\omega = x^\rota (x^\rota)^\transp$.

For any internal edge $E$, connecting the central vertex of $\calR(T)$ with one of the vertices of $T$, with tangent vector $\tau$ and normal vectror $\nu$, the matrix field $\omega$ has the property that, on the edge,  $\omega \tau = 0$ and $\partial_\nu \omega \tau \cdot \tau = 0$. See Lemma \ref{lem:kosreg}. Furthermore on any (external) edge $E$ of $T$, $\omega \tau \cdot \tau \in \rmP^0(E)$, $\omega \tau \cdot \nu \in \rmP^1(E)$ and $\partial_\nu \omega \tau \cdot \tau \in \rmP^0(E)$. Finally $\sven \omega$ is constant. 

We can define $W(T) = \{u \omega \ : u \in \rmP^0(\calR(T), \bbR) \}$. Indeed such $u \omega$ will have restriction $0$ to internal edges, so the $\sven$ can be computed classically and is proportional to $u$.
\end{remark}

\begin{remark}
Discrete diagram chase for the finite element strain complex with lower regularity. See Figure \ref{fig:diagstrainlow}.
\end{remark}

\begin{figure}
\includegraphics[width= \textwidth]{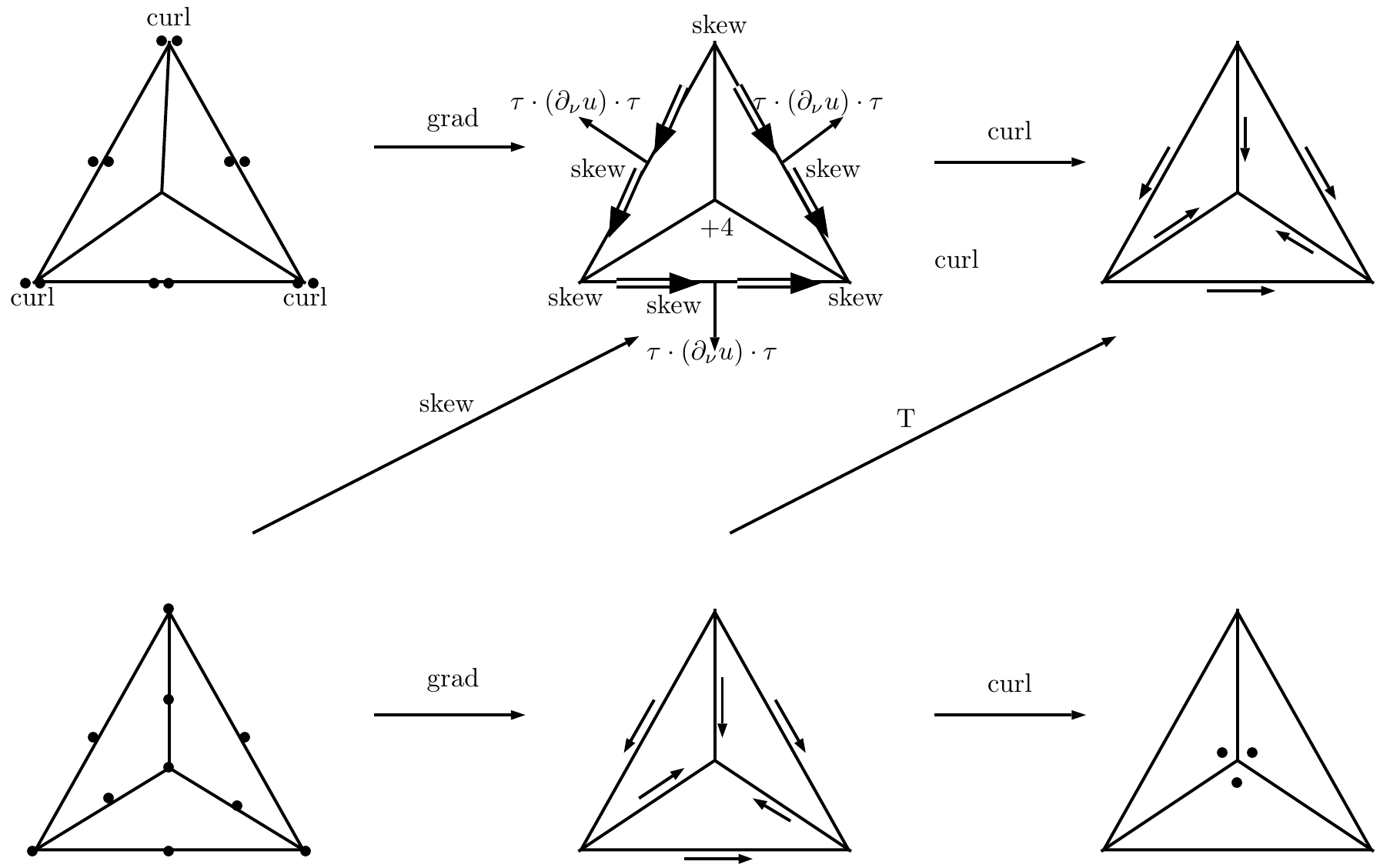}
\caption{Diagram chase for the strain complex with lower regularity. \label{fig:diagstrainlow}}
\end{figure}

\begin{remark}[minimal finite element strain complex with low regularity]
We get a minimal complex by imposing all DoFs of the type, on edges $E$, integral against $\chi_E$ to be zero. The dimensions are then $9$ for the vector fields, $9$ for the strain tensors and $3$ for the scalars.

Notice that then $A^1(E) = \rmP^0(E) \times \rmP^0(E) \times \rmP^0(E)$. We also have:
\begin{equation}
A^0(E) = \{(u,v,u')\in \rmP^1(E) \times \rmP^2(E) \times \rmP^1(E) \ : \ u' + \partial_\tau v \ \textrm{ is constant} \}.
\end{equation}
\end{remark}

\section*{Acknowledgements}
This paper puts together some ideas that have matured over a long time.

We are grateful to J\"org Frauendiener for interesting discussions on discrete connections and in particular Bianchi identities, a decade ago. Richard S. Falk provided us with reference \cite{ArnDouGup84}, which proved stimulating. Jean-Claude N\'ed\'elec, Douglas Arnold and Ragnar Winther have kindly shared their expertise, on mixed finite elements in general and elasticity complexes in particular. We also thank John Rognes, Geir Ellingsrud, Hans Munthe-Kaas and Jean Fran\c cois Pommaret, for their useful inputs related to de Rham theorems, vector bundles, inverse limits and Spencer sequences. Kristin Shaw recently pointed out possible connections with \cite{Cur14}.

Part of this work was carried out while SHC was supported by the European Research Council through the FP7-IDEAS-ERC Starting Grant scheme, project 278011 STUCCOFIELDS, including visits to ENS and IHP in Paris in 2015.

The research of KH leading to the results of this paper was partly carried out during his affiliation with the University of Oslo.  KH was then supported by the  European Research Council under the European Union's Seventh Framework Programme (FP7/2007-2013) / ERC grant agreement 339643 (FEEC-A).

\bibliography{../Bibliography/alexandria,../Bibliography/newalexandria,../Bibliography/mybibliography}{}
\bibliographystyle{plain}

\end{document}